\documentclass{amsart}

\usepackage{amsmath} 
\usepackage{amssymb}
\usepackage{bm}

\newcount\skewfactor
\def\mathunderaccent#1#2 {\let\theaccent#1\skewfactor#2
\mathpalette\putaccentunder}
\def\putaccentunder#1#2{\oalign{$#1#2$\crcr\hidewidth
\vbox to.2ex{\hbox{$#1\skew\skewfactor\theaccent{}$}\vss}\hidewidth}}
\def\name{\mathunderaccent\tilde-3 }

\newcommand{\forces}{\Vdash} 
\newcommand{\con}{{\mathfrak c}}

\newcommand{\can}{{}^{\omega}2} 

\newcommand{\lh}{\ell g\/} 
\newcommand{\rest}{{\restriction}}
 
\newcommand{\rng}{{\rm rng}}
\newcommand{\conc}{{}^\frown\!}
\newcommand{\vtl}{\vartriangleleft} 
\newcommand{\vare}{\varepsilon}

\newcommand{\cf}{{\rm cf}}
\newcommand{\rk}{{\rm rk}}
\newcommand{\rksp}{{\rm rk}^{\rm sp}}

\newcommand{\stnd}{{\rm stnd}}
\newcommand{\std}{{\rm std}}
\newcommand{\ndrk}{{\rm ndrk}_\iota}

\newcommand{\NDRK}{{\rm ndrk}_\iota}
\newcommand{\Mtk}{{{\mathbf M}_{\bar{T},\iota}}}
\newcommand{\Mtks}{{{\mathbf M}_{\bar{T}^*,\iota}}}
\newcommand{\fMtk}{{{\mathbf M}^n_{\bar{t},\iota}}}
\newcommand{\fMtp}{{{\mathbf M}^{n^p}_{\bar{t}^p,\iota}}}

\newcommand{\YZR}{{\mathcal{YZR}}}

\newcommand{\cA}{{\mathcal A}}

\newcommand{\cB}{{\mathcal B}}
\newcommand{\gb}{{\mathfrak b}}

\newcommand{\bj}{{\mathbf j}}

\newcommand{\bk}{{\mathbf k}}

\newcommand{\cL}{{\mathcal L}}

\newcommand{\cM}{{\mathcal M}}

\newcommand{\bbM}{{\mathbb M}}
\newcommand{\bmm}{{\mathbf m}}
\newcommand{\bn}{{\mathbf n}}
\newcommand{\cP}{{\mathcal P}}
\newcommand{\bbP}{{\mathbb P}}

\newcommand{\mbR}{{\mathbb R}}

\newcommand{\cS}{{\mathcal S}}

\newcommand{\bV}{{\mathbf V}}

\newcommand{\bbZ}{{\mathbb Z}}

\newtheorem{theorem}{Theorem}[section] 
\newtheorem{claim}{Claim}[theorem]
\newtheorem{lemma}[theorem]{Lemma} 
\newtheorem{proposition}[theorem]{Proposition} 
\newtheorem{corollary}[theorem]{Corollary} 
\newtheorem{observation}[theorem]{Observation} 

\theoremstyle{definition}
\newtheorem{problem}[theorem]{Problem} 
 
\newtheorem{definition}[theorem]{Definition}
\newtheorem{hypothesis}[theorem]{Assumption}
\newtheorem{example}[theorem]{Example}

\theoremstyle{remark}

\newtheorem{remark}[theorem]{Remark}

\begin{document}

\title{Borel sets without perfectly many overlapping translations, II}

\author{Andrzej Ros{\l}anowski}
\address{Department of Mathematics\\
University of Nebraska at Omaha\\
Omaha, NE 68182-0243, USA}
\email{aroslanowski@unomaha.edu}

\author{Saharon Shelah}
\address{Institute of Mathematics\\
 The Hebrew University of Jerusalem\\
 91904 Jerusalem, Israel\\
 and  Department of Mathematics\\
 Rutgers University\\
 New Brunswick, NJ 08854, USA}
\email{shelah@math.huji.ac.il}
\urladdr{http://shelah.logic.at}

\thanks{Publication 1170 of the second author.\\
The first author thanks the National Science Fundation for supporting his
visit to Rutgers University where this research was carried out, and the 
Rutgers University for their hospitality.\\
Saharon Shelah thanks the Israel Science Foundation for their grant
1838/19.\\
Both authors would like to thank the referees for their valuable
comments which helped to improve the manuscript.}     

\date{September, 2021}

\begin{abstract}
  For a countable ordinal $\vare$ and an integer $\iota\geq 2$ we construct
  a $\Sigma^0_2$ subset of the Cantor space $\can$ for which one may force
  $\aleph_\vare$ translations with intersections of size $\geq 2\iota$, but
  such that it has no perfect set of such translations in any ccc
  extension. These sets have uncountably many translations with
  intersections of size $\geq 2\iota$ in ZFC, so this answers \cite[Problem
  3.4]{RoRy18}.
\end{abstract}

\maketitle 

\section{Introduction}
The existence of Borel sets with large squares but no perfect squares
was studied and resolved in Shelah \cite{Sh:522}. We say that a set
$B\subseteq \can\times\can$ contains a $\mu$--square (perfect square,
respectively), if there is a set $X$ of cardinality $\mu$ (a perfect
set $X$, respectively) such that $X\times X\subseteq B$. It was shown
in \cite[Section 1]{Sh:522} that  
\begin{quotation}
it is consistent that for every ordinal $\alpha<\omega_1$, there is a Borel
subset of $\can\times\can$ containing an $\aleph_\alpha$--square but no
perfect square.    
\end{quotation}
As a matter of fact the problem was given a more complete answer. A rank on   
models in a countable vocabulary (called here {\em a splitting rank\/}, see
Definition \ref{defofrank}) occured to be closely related to the question
when we can force $\Sigma^0_2$ sets with $\mu$--squares but without perfect
squares.  The first $\lambda$,  called $\lambda_{\omega_1}$, such that there
is no model with universe $\lambda$, countable vocabulary and countable rank
is a cutting point here. Every $\Sigma^1_1$ set containing a
$\lambda_{\omega_1}$--square must contain a perfect square. On the other
hand for each cardinal $\mu< \lambda_{\omega_1}$ some ccc forcing notion
adds a $\Sigma^0_2$ set containing a $\mu$--square but no perfect
square. The cardinal  $\lambda_{\omega_1}$ is quite mysterious: it satisfies 
$\aleph_{\omega_1} \leq \lambda_{\omega_1}\leq \beth_{\omega_1}$ and 
(its close relative) cannot be increased by ccc forcing, but not much
more is known. 

Thinking about subsets of the plane as relations, one may wonder for what
kinds of relations we have similar results. Several questions may be reduced to
the existence of large squares for special kinds of Borel subsets of
$\can\times \can$. For instance, for $A\subseteq \can$ let the {\em spectrum
  of translation $k$--disjointness of $A$\/} be defined as    
\[\std_k (A)=\{(x,y)\in\can\times\can: |(A+x)\cap (A+y)|\leq k\}.\]
Then a $\mu$--square included in $\std_k(A)$ corresponds to a family of
$\mu$ many translations of $A$ with pairwise intersections 
of  size $\leq k$, and for $k=0$ this would be $\mu$ many pairwise 
disjoint translations. Interest in Borel sets with $\mu\geq \aleph_1$
pairwise disjoint translations but without any perfect set of such
translations is motivated by several works in literature. For instance,
Balcerzak, Ros{\l}anowski and  Shelah \cite{BRSh:512} studied the
$\sigma$--ideal of subsets of $\can$ generated by Borel sets with a perfect
set of pairwise disjoint translations. A generalization of this direction
follows Darji and Keleti \cite{UdKe03}, Elekes and Stepr\={a}ns
\cite{ElSt04},  and Zakrzewski \cite{Zak13}. They studied perfectly
$k$--small sets which for a finite $k$ can be described as follows. A set
$A\subseteq \can$ is perfectly $k$--small if there is a perfect set
$P\subseteq \can$ such that for distinct $x_0,\ldots,x_{k-1}\in P$ the
intersection $(A+x_0)\cap\ldots\cap (A+x_{k-1})$ is empty. Elekes and Keleti
\cite{ElKe15} studied decompositions of the real line into pairwise disjoint
Borel pieces so that each piece is closed under addition and in this context
they explicitly asks \cite[Question 4.5]{ElKe15}:
\begin{quotation}
Suppose that a Borel subset of $\mbR$ has uncountably many pairwise disjoint
translates. Does it also have continuum many pairwise disjoint translates?   
\end{quotation}
If we want to answer the above question by a direct application of
\cite[Section 1]{Sh:522}, we could look for a $\Sigma^0_2$ set $A\subseteq 
\can$ such that $\std_0(A)$ contains a large square but no very large
square. However, in this situation, $\std_0(A)$ is a $\Pi^0_2$ subset of
$\can\times \can$ and, as it was noted in \cite[Remark 1.14]{Sh:522},
\begin{quotation}
if $B\subseteq \can\times \can$ is a $\Pi^0_2$ set and it contains an 
uncountable square, then it contains a perfect square.   
\end{quotation}
Therefore, forcing ``a bad Borel set'' for $\std_k$ must involve adding a 
$\Pi^0_2$ (or more complex) subset of $\can$, a task that at the moment
appears substantially more complicated than adding ``a bad $\Sigma^0_2$
set''. 

In developing tools to deal with $\std_k$ and perfect sets of disjoint
translations, we looked into the dual direction. Now, for a set $A\subseteq
\can$ and an integer $k$ we consider the {\em spectrum of  translation
  $k$--non-disjointness\/} of $A$,   
\[\stnd_k(A)=\{(x,y)\in\can\times\can: |(A+x)\cap (A+y)|\geq k\}.\] 
Then a $\mu$--square included in  $\stnd_k(A)$ determines a family of
$\mu$ many pairwise $k$--overlapping translations. The existence of
Borel sets with many, but not too many, pairwise $k$--overlapping
translations was studied in Ros{\l}anowski and Rykov \cite{RoRy18} and
Ros{\l}anowski and Shelah \cite{RoSh:1138}. In the latter work we carried
out arguments fully parallel to that of \cite[Section 1]{Sh:522} and we
showed that, e.g.,  for $\lambda<\aleph_{\omega_1}$ and an even integer
$k\geq 6$ there is a ccc forcing notion $\bbP$ adding a $\Sigma^0_2$  set
$B\subseteq\can$ with the property that   
\begin{itemize}
\item for some $H\subseteq \can$ of size $\lambda$, $|(B+h)\cap (B+h')|\geq 
  k$ for all $h,h'\in H$, but 
\item for every perfect set $P\subseteq \can$ there are $x,x'\in P$ with 
$|(B+x)\cap (B+x')|<k$. 
\end{itemize}
\bigskip

Our goal in the current article is to analyze the construction of
\cite{RoSh:1138} and split it into two steps: first constructing a
$\Sigma^0_2$ set (in ZFC) and then forcing non-disjoint translations to this
set. (A similar analysis for homogeneous sets for analytic colorings was
done by Kubi\'s and Shelah \cite{KbSh:802}.) In addition to better
understanding of the connection between the splitting rank and forcing
non-disjoint translations, we get an improvement over the older results,
extending them to even integers $k\geq 4$. Moreover, our analysis allows us
to answer  \cite[Problem 3.4]{RoRy18}:  there are $\Sigma^0_2$ subsets of
$\can$ with uncountably many pairwise 4--non-disjoint translations but with
no perfect set of such translations (cf Corollary \ref{answ}). In relation to
that problem, let us give an easy construction of a $\Sigma^0_2$ set
$B^*\subseteq\can \times\can$ containing an uncountable square but no
perfect square. This set, however, does not work for \cite[Problem
3.4]{RoRy18} as it is not of the form $\stnd_k(A)$.  
\bigskip

Fix a bijection $\pi:\omega\times\omega\longrightarrow\omega$ and define a 
set $B^*\subseteq \can\times\can$ as follows: 

\[\begin{array}{ll}
    (x,y)\in B^* \Leftrightarrow &x=y\ \vee\ \big(\exists k\in \omega \big) 
    \big (\forall n\in\omega \big) \big (  x(n)=y(\pi(n,k)) \big)\ \vee\\
    &\qquad \big (\exists k\in \omega \big) \big (\forall n\in \omega \big) 
      \big ( y(n)=x(\pi(n,k)) \big). 
    \end{array}\]

  \begin{proposition}
    \begin{enumerate}
\item There is an uncountable set $X\subseteq \can$ such that $X\times 
      X\subseteq B^*$. 
\item There is no perfect set $P\subseteq \can$ such that $P\times P 
      \subseteq B^*$. 
    \end{enumerate}
  \end{proposition}

  \begin{proof}
(1)\quad We choose inductively a sequence $\langle 
x_\alpha:\alpha<\omega_1\rangle$ of distinct elements of $\can$ satisfying 
\begin{enumerate}
\item[$(\boxtimes)$] $\alpha<\beta<\omega_1\quad\Rightarrow\quad x_\alpha 
  \neq x_\beta\ \wedge\ \big  (\exists k\in \omega\big) \big (\forall n\in 
  \omega \big) \big(x_\alpha(n)= x_\beta(\pi(n,k)) \big)$.   
\end{enumerate}
So, when arriving to stage $\beta<\omega_1$, we first choose a sequence
$\langle  y_k:k<\omega\rangle\subseteq \can$  so that 
\begin{itemize}
\item $\{x_\alpha:\alpha<\beta\}\subseteq \{y_k:k<\omega\}$, and 
\item $\big(\forall \alpha<\beta \big) \big (\exists n<\omega \big) \big 
  (y_0(n)\neq x_\alpha(\pi(n,0)) \big)$ . 
\end{itemize}
Next we define 
\[x_\beta(i)=y_k(n)\mbox{ whenever }  i=\pi(n,k),\]
Note that $x_\beta$ satisfies the demand in $(\boxtimes)$ (for 
$\alpha<\beta$). 

After the inductive construction is completed, it should be clear that the 
set $X=\{x_\alpha: \alpha<\omega_1\}$ is uncountable and $X\times X
\subseteq B^*$. 
\medskip 

\noindent (2)\quad Assume towards contradiction that $P\subseteq \can$ is a 
perfect set such that $P\times P\subseteq B^*$. For $k<\omega$ let 
\[\begin{array}{ll}
R_{2k}=&\{(x,y)\in\can\times\can: \big(\forall 
         n\in \omega \big)\big ( y(n)=x(\pi(n,k)) \big)\},\\
R_{2k+1}=&\{(x,y)\in\can\times\can:\big(\forall 
           n\in \omega \big)\big ( x(n)=y(\pi(n,k)) \big)\}.    
\end{array}\] 
These are closed sets and $P\times P\subseteq \bigcup\limits_{\ell<\omega}
R_\ell\cup\{(x,x):x\in P\}$, so by Mycielski theorem \cite[Theorem 1, 
p. 141]{My64} (see also \cite[Lemma 2.4]{RoRy18}), there are a perfect set 
$P'\subseteq P$ and an increasing sequence of integers 
$0=n_0<n_1<n_2<n_3<\ldots$  such that 
\begin{enumerate}
\item[$(\heartsuit)$] for each $k<\omega$, $x,x',y,y'\in P'$ and $\ell\leq 
  k$, if $x'\rest n_k =x\rest n_k\neq y\rest n_k=y'\rest n_k$, then 
\[(x,y)\in R_\ell\ \Leftrightarrow\ (x',y')\in  R_\ell.\]  
\end{enumerate}
Take distinct $x,y\in P'$ and let $\ell$ be such that $(x,y)\in R_\ell$; by 
symmetry we may assume that $\ell$ is even, say $\ell=2i$. Choose $k>\ell$
such that $x\rest n_k\neq  y\rest n_k$ and pick $y'\in P'$ such 
that $y\neq y'$ and $y\rest n_k=y'\rest n_k$. It follows from $(\heartsuit)$
that $(x,y')\in R_\ell=R_{2i}$ and hence $y'(n)=x(\pi(n,i))=y(n)$ for all 
$n\in \omega$, a contradiction.  
\end{proof}
\bigskip

Every uncountable Borel subset $B$ of $\can$ has a perfect set of pairwise
non-disjoint translations (just consider a perfect set $P\subseteq B$ and
note that for $x,y\in P$ we have ${\mathbf 0},x+y\in (B+x)\cap (B+y)$). The
problem of many non-disjoint translations is more interesting if we demand
that the intersections have more elements. Note  that in $\can$, if
$x+b_0=y+b_1$ then also $x+b_1=y+b_0$. Consequently, if $x\neq y$ and
$|(B+x)\cap (B+y)|<\omega$, then $|(B+x)\cap (B+y)|$ is even.   
Therefore we will look at intersections of size $\geq 2\iota$ and (unlike in
\cite{RoSh:1138}) we will manage to deal here with any finite $\iota\ge 2$.  
\bigskip

We fully utilize the algebraic properties of $(\can,+)$, in particular the
fact that all elements of $\can$ are self-inverse. Independence
results for the general case of Abelian Polish groups is investigated
in the third paper of the series \cite{RoSh:1187}, however we do not carry out
any rank analysis there (leaving that aspect open). 
\bigskip

\noindent{\bf Notation}:\qquad Our notation is rather standard and
compatible with that of classical textbooks (like Jech \cite{J} or
Bartoszy\'nski and Judah \cite{BaJu95}). However, in forcing we keep the
older convention that {\em a stronger condition is the larger one}.

\begin{enumerate}
\item For a set $u$ we let 
\[u^{\langle 2\rangle}=\{(x,y)\in u\times u:x\neq y\}.\]
\item For two sequences $\eta,\nu$ we write $\nu\vtl\eta$ whenever
  $\nu$ is a proper initial segment of $\eta$, and $\nu \trianglelefteq\eta$
  when either $\nu\vtl\eta$ or $\nu=\eta$.  
\item The set of all sequences of length $n$ and with values in $\{0,1\}$ is
denoted by ${}^n 2$ and we let ${}^{\omega>} 2=\bigcup\limits_{n<\omega}
{}^n 2$.
\item The Cantor space $\can$ of all infinite sequences with values 0 and 1 
is equipped with the natural product topology and the group operation of 
coordinate-wise addition $+$ modulo 2.   
\item A {\em tree\/} is a $\vtl$--downward closed set of sequences.  For a
  tree $T\subseteq {}^{\omega>} 2$ the set of all $\omega$--branches through
  $T$ is denoted $\lim(T)$. 
\item Ordinal numbers will be denoted by lower case initial letters of
  the Greek alphabet $\alpha,\beta,\gamma,\delta,\vare$. Finite ordinals
  (non-negative integers) will be denoted by letters
  $a,b,c,d,i,j,k,\ell,m,n,J,K,L,M,N$ and $\iota$. For integers $N^s<N^t$, 
  notations of the form $[N^s,N^t)$ are used to denote {\em intervals of  
    integers}. 
\item The Greek letter $\lambda$ will stand for an uncountable cardinal. 
\item For a forcing notion $\bbP$, all $\bbP$--names for objects in
  the extension via $\bbP$ will be denoted with a tilde below (e.g.,
  $\name{\tau}$, $\name{X}$), and $\name{G}_\bbP$ will stand for the
  canonical $\bbP$--name for the generic filter in $\bbP$.
\end{enumerate}

\section{Two Ranks from the Past}
Let us recall two closely related ranks used in previous papers. They are central
for the studies here too. 

\subsection{Splitting rank $\rksp$}
The results recalled in this subsection are quoted from \cite[Section
2]{RoSh:1138}, however they were first given in \cite[Section 1]{Sh:522}. 

Let $\lambda$ be a cardinal and $\bbM$ be a model with the universe
$\lambda$ and a countable vocabulary $\tau$.

\begin{definition}
\label{defofrank}
\begin{enumerate}
\item By induction on ordinals $\delta$, for finite non-empty sets
  $w\subseteq\lambda$ we define when $\rk(w,\bbM)\geq \delta$. Let
  $w=\{\alpha_0,\ldots,\alpha_n\} \subseteq\lambda$, $|w|=n+1$.  
 \begin{enumerate}
\item[(a)] $\rk(w, \bbM)\geq 0$ if and only if for every quantifier free
  formula $\varphi\in \cL(\tau)$ and each $k\leq n$, if  $\bbM\models
  \varphi[\alpha_0,\ldots,\alpha_k,\ldots,\alpha_n]$ then  the set  
\[\big\{\alpha\in \lambda:\bbM\models \varphi[\alpha_0,\ldots,\alpha_{k-1},
  \alpha,\alpha_{k+1}, \ldots,\alpha_n]\big\}\]  is 
uncountable;  
\item[(b)] if $\delta$ is limit, then $\rk(w,\bbM)\geq\delta$ if and only if 
  $\rk(w,\bbM)\geq\gamma$ for all $\gamma<\delta$;  
\item[(c)] $\rk(w,\bbM)\geq\delta+1$ if and only if for every quantifier
  free  formula $\varphi\in \cL(\tau)$ and each $k\leq n$, if $\bbM\models  
  \varphi[\alpha_0,\ldots,\alpha_k,\ldots,\alpha_n]$ then there is
  $\alpha^*\in\lambda\setminus w$ such that 
\[\rk(w\cup\{\alpha^*\},\bbM)\geq \delta\quad\mbox{ and }\quad \bbM\models  
  \varphi[\alpha_0,\ldots,\alpha_{k-1},\alpha^*,\alpha_{k+1},\ldots,\alpha_n].\] 
 \end{enumerate}
\end{enumerate}
\end{definition}

By a straightforward induction on $\delta$ one easily shows that if 
$\emptyset\neq v\subseteq w$ then 
\[\rk(w,\bbM)\geq\delta\geq\gamma \implies \rk(v,\bbM)\geq \gamma.\]
Hence we may define the rank functions on finite non-empty subsets of
$\lambda$. 

\begin{definition}
The rank $\rk(w,\bbM)$ of a finite non-empty set $w\subseteq\lambda$
is defined as:  
 \begin{itemize}
\item $\rk(w,\bbM)=-1$ if $\neg (\rk(w,\bbM)\geq 0)$,
\item $\rk(w,\bbM)=\infty$ if $\rk(w,\bbM)\geq \delta$ for all ordinals
  $\delta$,
\item for an ordinal $\delta$: $\rk(w,\bbM)=\delta$ if $\rk(w,\bbM)\geq
  \delta$ but $\neg(\rk(w,\bbM)\geq\delta+1)$.
\end{itemize}
\end{definition}

\begin{definition}
  \label{PRdef}
For an ordinal $\vare$ and a cardinal $\lambda$ let ${\rm
  NPr}^\vare(\lambda)$  be the following statement:
\begin{quotation}
``there is a model  $\bbM^*$ with the universe $\lambda$ and a countable
vocabulary $\tau^*$ such  that  $1+\rk(w,\bbM^*)\leq\vare$ for all $w\in
[\lambda]^{<\omega}\setminus\{\emptyset\}$.''      
\end{quotation}
Let ${\rm Pr}^\vare(\lambda)$ be the  negation of ${\rm NPr}^\vare(\lambda)$.
\end{definition}
(Note that ${\rm NPr}_\vare$ of \cite[Definition 2.4]{RoSh:1138}
differs from our ${\rm NPr}^\vare$: ``$\sup\{\rk(w,\bbM^*):
\emptyset\neq w\in [\lambda]^{<\omega} \}<\vare$  '' there is replaced
by ``$1+\rk(w,\bbM^*)\leq\vare$'' here.) 

\begin{proposition}
\label{cl1.7-522}
\begin{enumerate}
\item ${\rm NPr}^1(\aleph_1)$.
\item If ${\rm NPr}^\vare(\lambda)$, then ${\rm NPr}^{\vare+1}(\lambda^+)$.  
\item If ${\rm NPr}^\vare(\mu)$ for $\mu<\lambda$ and
  $\cf(\lambda)=\omega$, then ${\rm NPr}^\vare(\lambda)$. 
\item If $\alpha<\omega_1$, then ${\rm NPr}^{\alpha}(\aleph_\alpha)$ but  
${\rm Pr}^\alpha(\beth_{\omega_1})$ holds. 
\end{enumerate}
\end{proposition}

\begin{definition}
Let $\tau^\otimes=\{R_{n,j}:n,j< \omega\}$ be a fixed relational
vocabulary where $R_{n,j}$ is an $n$--ary relational symbol (for
$n,j<\omega$).  
\end{definition}

\begin{definition}
\label{hypo2}
Assume that $\vare<\omega_1$ and $\lambda$ is an uncountable cardinal such
that  ${\rm NPr}^\vare(\lambda)$. By this assumption, we may fix a model
$\bbM(\vare,\lambda)= \bbM=(\lambda, \{R^\bbM_{n,j}\}_{n,j<\omega}) $ in the
vocabulary $\tau^\otimes$ with the  universe $\lambda$ such that:  
\begin{enumerate}
\item[$(\circledast)_{\rm a}$] for every $n$ and a quantifier free formula 
  $\varphi(x_0,\ldots,x_{n-1})\in \cL(\tau^\otimes)$ there is $j<\omega$ such
  that for all $\alpha_0,\ldots, \alpha_{n-1}\in \lambda$, 
\[\bbM\models\varphi[\alpha_0,\ldots,\alpha_{n-1}]\Leftrightarrow
  R_{n,j}[\alpha_0, \ldots, \alpha_{n-1}],\]
\item[$(\circledast)_{\rm b}$] the rank of every singleton is at least 0,  
\item[$(\circledast)_{\rm c}$] $1+\rk(v,\bbM)\leq\vare$ for every $v\in
  [\lambda]^{<\omega} \setminus \{\emptyset\}$.
\end{enumerate}
For a nonempty finite set $v\subseteq\lambda$ let
$\rk^{\rm sp}(v)=\rk(v,\bbM)$, and let $\bj(v)<\omega$ and
$\bk(v)<|v|$ be such that $R_{|v|,\bj(v)},\bk(v)$ witness
the rank of $v$. Thus letting $\{\alpha_0,\ldots, \alpha_k, \ldots
\alpha_{n-1}\}$ be the increasing enumeration of $v$ and $k=\bk(v)$ and
$j= \bj(v)$, we have 
\begin{enumerate}
\item[$(\circledast)_{\rm e}$] if $\rk^{\rm sp}(v)\geq 0$, then
  $\bbM\models R_{n,j}[\alpha_0,\ldots, \alpha_k,\ldots, \alpha_{n-1}]$
  but there is no $\alpha\in \lambda\setminus v$  such that  
\[\rk^{\rm sp}(v\cup\{\alpha\})\geq \rk^{\rm sp}(v)\ \mbox{ and }\ 
  \bbM\models R_{n,j} [\alpha_0,\ldots, \alpha_{k-1}, \alpha,
  \alpha_{k+1}, \ldots,\alpha_{n-1}],\]  
\item[$(\circledast)_{\rm f}$] if $\rk^{\rm sp}(v)=-1$, then
  $\bbM\models R_{n,j} [\alpha_0,\ldots,\alpha_k,\ldots, \alpha_{n-1}]$
  but the set
  \[\big\{\alpha\in\lambda:\bbM\models R_{n,j}[\alpha_0,\ldots,
    \alpha_{k-1},\alpha, \alpha_{k+1}, \ldots, \alpha_{n-1}]\big\}\] 
is countable.  
\end{enumerate}
We may and will also require that for $j=\bj(v)$, $n=|v|$ we have: 
\begin{enumerate}
\item[$(\circledast)_{\rm g}$] for every
  $\beta_0,\ldots,\beta_{n-1}<\lambda$  
\[\mbox{if }\ \bbM\models R_{n,j}[\beta_0,\ldots,\beta_{n-1}] \mbox{
    then }\  \beta_0<\ldots <\beta_{n-1}.\]  
\end{enumerate}
The choices above define functions $\bj:[\lambda]^{<\omega}\setminus \{
\emptyset\} \longrightarrow \omega$, $\bk:[\lambda]^{<\omega}\setminus  
\{\emptyset\} \longrightarrow \omega$, and $\rksp:[\lambda]^{<\omega}
\setminus \{\emptyset\} \longrightarrow \{-1\}\cup (\vare+1)$. 
\end{definition}

\subsection{Non-disjontness rank $\ndrk$}
Here we recall the rank measuring the easiness of building large sets of 
pairwise overlapping translations of a given $\Sigma^0_2$ set. The
definitions and results given here are  quoted after \cite[Section
3]{RoSh:1138}. Let us point out that Definition \ref{mtkDef} is a
slightly modified version of \cite[Definition 3.5]{RoSh:1138} -- we
added demand (f) here. The addition is needed for the precise rank
considerations when our ranks are finite (to eliminate ``disturbances
in rank'' by not important factors). It does not change the proofs of
the facts quoted here, however. 

We assume the following. 

\begin{hypothesis}
\label{hyp}
  \begin{enumerate}
\item $T_n\subseteq {}^{\omega>} 2$ is a tree with no maximal nodes
    (for $n<\omega$);
\item $B=\bigcup\limits_{n<\omega} \lim(T_n)$, $\bar{T}=\langle T_n:
  n<\omega \rangle$ and $2\leq\iota<\omega$;
\item there are distinct $\rho_0,\rho_1\in \can$ such that
  $\big|(\rho_0+B)\cap (\rho_1+B)\big|\geq 2\iota$.
  \end{enumerate}
\end{hypothesis}

\begin{definition}
  \label{mtkDef}
Let $\Mtk$ consist of all tuples 
\[\bmm=(\ell_\bmm,u_\bmm,\bar{h}_\bmm,\bar{g}_\bmm)=(\ell,u,\bar{h},\bar{g})\]
such that:
\begin{enumerate}
\item[(a)] $0<\ell<\omega$, $u\subseteq {}^\ell 2$ and $2\leq |u|$;
\item[(b)] $\bar{h}=\langle h_i:i<\iota\rangle$, $\bar{g}=\langle g_i:i<
  \iota\rangle$ and for each $i<\iota$ we have 
\[h_i:u^{\langle 2\rangle}\longrightarrow \omega\quad\mbox{ and }\quad 
    g_i:u^{\langle 2\rangle} \longrightarrow \bigcup_{n<\omega}( T_n\cap
    {}^\ell 2);\]
\item[(c)] $g_i(\eta,\nu)\in T_{h_i(\eta,\nu)}\cap {}^\ell 2$ for all
  $(\eta,\nu)\in u^{\langle 2\rangle}$, $i<\iota$;
\item[(d)] if $(\eta,\nu)\in u^{\langle 2\rangle}$ and $i<\iota$, then
  $\eta+ g_i(\eta,\nu) =\nu+ g_i(\nu,\eta)$;  
\item[(e)] for any $(\eta,\nu)\in u^{\langle 2\rangle}$, there are no
  repetitions in the sequence $\langle g_i(\eta,\nu),g_i(\nu,\eta):
  i<\iota\rangle$;
\item[(f)] there are $\langle F(\eta):\eta\in u\rangle$ and $\langle
  G_i(\eta,\nu): i<\iota\ \wedge\ (\eta,\nu)\in u^{\langle 2\rangle}
  \rangle$ such that 
 \[\begin{array}{r}
\eta\vtl F(\eta)\in\can\ \mbox{ and }\ g_i(\eta,\nu)\vtl G_i(\eta,\nu)\in
   \lim\big(T_{h_i(\eta,\nu)}\big)\\
\mbox{ and }\  F(\eta)+ G_i(\eta,\nu)  =F(\nu)+ G_i(\nu,\eta)
\end{array}\]
 (for $i<\iota$, $(\eta,\nu)\in u^{\langle 2\rangle}$). 
\end{enumerate}
\end{definition}

Note that by Assumption \ref{hyp}(3) the family $\Mtk$ is not empty. 

\begin{definition}
\label{traDef}
Assume $\bmm=(\ell,u,\bar{h},\bar{g})\in \Mtk$ and $\rho\in {}^\ell 2$. We
define $\bmm+\rho=(\ell',u',\bar{h}',\bar{g}')$ by
\begin{itemize}
\item $\ell'=\ell$, $u'=\{\eta+\rho:\eta\in u\}$,
\item $\bar{h}'=\langle h'_i:i<\iota\rangle$ where $h'_i:(u')^{\langle
    2\rangle} \longrightarrow \omega$ are such that
  $h'_i(\eta+\rho,\nu+\rho)=h_i(\eta,\nu)$ for $(\eta,\nu)\in u^{\langle
    2\rangle}$, 
\item $\bar{g}'=\langle g'_i:i<\iota\rangle$ where $g'_i:(u')^{\langle 2\rangle}
  \longrightarrow \bigcup\limits_{n<\omega} (T_n\cap {}^\ell 2)$ are such
  that 
\[g'_i (\eta+\rho, \nu+\rho)=g_i(\eta,\nu)\ \mbox{ for }(\eta,\nu)\in
  u^{\langle 2\rangle}.\] 
\end{itemize}
Also if $\rho\in\can$, then we set $\bmm+\rho=\bmm+(\rho\rest\ell)$. 
\end{definition}

\begin{observation}
  \begin{enumerate}
\item If $\bmm\in \Mtk$ and $\rho\in {}^{\ell_\bmm}2$, then $\bmm+\rho
  \in\Mtk$. 
\item For each $\rho\in\can$ the mapping
  $\Mtk\longrightarrow\Mtk:\bmm\mapsto\bmm+\rho$ is a bijection.
  \end{enumerate}
\end{observation}

\begin{definition}
\label{extDef}
Assume $\bmm,\bn\in\Mtk$. We say that {\em $\bn$ extends $\bmm$\/}
($\bmm\sqsubseteq \bn$ in short) if and only if:
\begin{itemize}
\item $\ell_\bmm\leq \ell_\bn$, $u_\bmm=\{\eta\rest\ell_\bmm:\eta\in u_\bn\}$,
  and 
\item for every $(\eta,\nu)\in (u_\bn)^{\langle 2\rangle }$ such that
  $\eta\rest \ell_\bmm \neq \nu\rest\ell_\bmm$ and each $i<\iota$ we have  
\[h^\bmm_i(\eta\rest \ell_\bmm,\nu\rest \ell_\bmm)= h^\bn_i(\eta,\nu)\quad
  \mbox{ and }\quad g^\bmm_i(\eta\rest \ell_\bmm,\nu\rest \ell_\bmm)=
  g^\bn_i(\eta,\nu)\rest \ell_\bmm.\] 
\end{itemize}
\end{definition}

\begin{definition}
\label{ndrkdef}
We define a function $\ndrk:\Mtk\longrightarrow {\rm ON}\cup\{\infty\}$
declaring inductively when $\ndrk(\bmm)\geq\alpha$ (for an ordinal
$\alpha$). 
\begin{itemize}
\item $\ndrk(\bmm)\geq 0$ always;
\item if $\alpha$ is a limit ordinal, then $\ndrk(\bmm)\geq \alpha
  \Leftrightarrow (\forall\beta<\alpha)(\ndrk(\bmm)\geq \beta)$;
\item if $\alpha=\beta+1$, then $\ndrk(\bmm)\geq \alpha$ if and only if for
  every $\nu\in u_\bmm$ there is $\bn\in \Mtk$ such that
  $\ell_\bn>\ell_\bmm$, $\bmm\sqsubseteq\bn$ and $\ndrk(\bn)\geq \beta$ and 
\[|\{\eta\in u_\bn:\nu\vtl\eta\}|\geq 2;\] 
\item $\ndrk(\bmm)=\infty$ if and only if $\ndrk(\bmm)\geq \alpha$ for all
  ordinals $\alpha$. 
\end{itemize}
We also define 
\[\NDRK(\bar{T})=\sup\{\ndrk(\bmm):\bmm\in\Mtk\}.\]
\end{definition}

\begin{lemma}
  [See {\cite[Lemma 3.10]{RoSh:1138}}]
  \label{lemonrk}
  \begin{enumerate}
\item The relation $\sqsubseteq$ is a partial order on $\Mtk$.
\item If $\bmm,\bn\in\Mtk$ and $\bmm\sqsubseteq\bn$ and $\alpha\leq
  \ndrk(\bn)$, then $\alpha\leq\ndrk(\bmm)$. 
\item The function $\ndrk$ is well defined.
\item If $\bmm\in\Mtk$ and $\rho\in\can$ then $\ndrk(\bmm)=\ndrk(\bmm+\rho)$. 
\item If $\bmm\in\Mtk$, $\nu\in u_\bmm$ and $\ndrk(\bmm)\geq \omega_1$, then
  there is an $\bn\in\Mtk$ such that $\bmm\sqsubseteq\bn$,
  $\ndrk(\bn)\geq\omega_1$, and  
\[|\{\eta\in u_\bn:\nu\vtl\eta\}|\geq 2.\]
\item If $\bmm\in\Mtk$ and $\infty>\ndrk(\bmm)=\beta>\alpha$, then there is
  $\bn\in \Mtk$ such that $\bmm\sqsubseteq \bn$ and $\ndrk(\bn)=\alpha$. 
\item If $\NDRK(\bar{T})\geq \omega_1$, then $\NDRK(\bar{T})=\infty$. 
\item Assume $\bmm\in\Mtk$ and $u'\subseteq u_\bmm$, $|u'|\geq 2$. Put
  $\ell'=\ell_\bmm$, $h_i'=h_i^\bmm\rest (u')^{\langle 2\rangle}$ and
  $g_i'=g_i^\bmm \rest (u')^{\langle 2\rangle}$ (for $i<\iota$), and let  $\bmm\rest
  u'=(\ell',u', \bar{h}', \bar{g}')$. Then $\bmm\rest u'\in \Mtk$
  and $\ndrk(\bmm)\leq \ndrk(\bmm\rest u')$. 
  \end{enumerate}
\end{lemma}

Directly from Definition \ref{ndrkdef} and Lemma \ref{lemonrk}(2) we get the
following observation. 

\begin{observation}
  \label{stepup}
 If $\bmm\in\Mtk$ and $\ndrk(\bmm)\geq \alpha+1$, then there is 
 $\bn\sqsupseteq \bmm$ such that $\ell_\bn=\ell_\bmm+1$ and $\ndrk(\bn)\geq 
 \alpha$. 
\end{observation}

\begin{proposition}
  [See {\cite[Proposition 3.11]{RoSh:1138}}]  
\label{eqnd}
The following conditions are equivalent.
\begin{enumerate}
\item[(a)] $\NDRK(\bar{T})\geq \omega_1$.
\item[(b)] $\NDRK(\bar{T})=\infty$.
\item[(c)] There is a perfect set $P\subseteq\can$ such that 
\[\big(\forall \eta,\nu\in P \big) \big (|(B+\eta)\cap (B+\nu)|\geq 2\iota
  \big).\] 
\end{enumerate}
\end{proposition}

\begin{proposition}
\label{3.11da}
Assume $\ndrk(\bar{T})< \vare$. If  there is a set $A\subseteq\can$ of  
cardinality $\lambda$ such that
\[\big(\forall \eta,\nu\in A\big) \big (|(B+\eta)\cap (B+\nu)|\geq 2\iota
  \big),\]
then  ${\rm NPr}^{1+\vare}(\lambda)$.
\end{proposition}

\begin{proof}
 This was implicitly shown by the proof of \cite[Proposition
 3.11$((d)\Rightarrow (a))$]{RoSh:1138}, but let us repeat this argument.  

Assume $\langle \eta_\alpha:\alpha<\lambda\rangle$ is a sequence of distinct
elements of $\can$ such that 
\[\big(\forall \alpha<\beta<\lambda\big)\big(|(B+\eta_\alpha)\cap 
  (B+\eta_\beta)|\geq 2\iota\big).\]
Let $\tau=\{R_\bmm:\bmm\in\Mtk\}$ be a (countable) vocabulary where each
$R_\bmm$ is a $|u_\bmm|$--ary relational symbol. Let $\bbM=\big(\lambda,
\big\{R^\bbM_\bmm\big\}_{\bmm\in \Mtk}\big)$ be the model in the vocabulary
$\tau$, where for $\bmm=(\ell,u,\bar{h}, \bar{g})\in\Mtk$  the relation
$R_\bmm^\bbM$ is defined by     
\[\begin{array}{ll}
R^\bbM_\bmm=&\Big\{(\alpha_0,\ldots,\alpha_{|u|-1})\in {}^{|u|}
\lambda:\{\eta_{\alpha_0}\rest \ell,\ldots,
        \eta_{|u|-1}\rest \ell\} =u \mbox{ and}\\
&\qquad \mbox{for distinct }j_1,j_2<|u|\mbox{ there are }
  G_i(\alpha_{j_1},\alpha_{j_2})\mbox{ (for $i<\iota$) such that}\\
&\qquad  g_i(\eta_{\alpha_{j_1}}\rest \ell,\eta_{\alpha_{j_2}}\rest \ell)
  \vtl G_i(\alpha_{j_1},\alpha_{j_2})\in \lim\big(
  T_{h_i(\eta_{\alpha_{j_1}}\rest \ell,\eta_{\alpha_{j_2}}\rest \ell)} \big)
  \mbox{ and}\\
&\qquad\eta_{\alpha_{j_1}}+G_i(\alpha_{j_1},\alpha_{j_2}) =
  \eta_{\alpha_{j_2}}+ G_i(\alpha_{j_2}, \alpha_{j_1})\ \Big\}.
\end{array}\]

We will show that the model $\bbM$ witnesses ${\rm
  NPr}^{1+\vare}(\lambda)$.

\begin{claim}
\label{cl12}
\begin{enumerate}
\item If $\alpha_0,\alpha_1,\ldots,\alpha_{j-1}<\lambda$ are
  distinct, $j\geq 2$, then for all sufficiently large $\ell<\omega$ there
  is $\bmm\in \Mtk$ such that  
\[\ell_\bmm=\ell,\quad u_\bmm=\{\eta_{\alpha_0}\rest \ell, \ldots,
\eta_{\alpha_{j-1}}\rest \ell\}\quad \mbox{ and }\quad \bbM\models
R_\bmm[\alpha_0,\ldots,\alpha_{j-1}].\]
\item Assume that  $\bmm\in\Mtk$, $j<|u_{\bmm}|$, 
  $\alpha_0,\alpha_1,\ldots,\alpha_{|u_{\bmm}|-1} < \lambda$ and
  $\alpha^*< \lambda$ are all pairwise distinct and such that 
  $\bbM\models R_{\bmm}[\alpha_0,\ldots,\alpha_j, \ldots,
  \alpha_{|u_{\bmm}|-1}]$   and  $\bbM\models R_{\bmm}[\alpha_0,\ldots,
  \alpha_{j-1},\alpha^*,\alpha_{j+1},  \ldots \alpha_{|u_{\bmm}|-1}]$. Then
  for every sufficiently large $\ell>\ell_{\bmm}$ there is $\bn\in \Mtk$
  such that $\bmm\sqsubseteq \bn$  and  
\[\ell_\bn=\ell,\quad u_\bn=\{\eta_{\alpha_0}\rest \ell, \ldots,
\eta_{\alpha_{|u_\bmm|-1}}\rest \ell,\eta_{\alpha^*}\rest \ell\}\quad \mbox{
  and 
}\quad \bbM\models R_\bn[\alpha_0,\ldots,\alpha_{|u_\bmm|-1},\alpha^*].\]
\end{enumerate}
\end{claim}

\begin{proof}[Proof of the Claim]
  (1)\quad For  distinct $j_1,j_2<j$ let $G_i(\alpha_{j_1},\alpha_{j_2})\in
  B$ (for $i<\iota$) be such that 
\[\eta_{\alpha_{j_1}}+G_i(\alpha_{j_1},\alpha_{j_2}) =
  \eta_{\alpha_{j_2}}+ G_i(\alpha_{j_2}, \alpha_{j_1})\]
and there are no repetitions in the sequence $\langle
G_i(\alpha_{j_1},\alpha_{j_2}), G_i(\alpha_{j_2},\alpha_{j_1}):
i<\iota\rangle$. (Remember, $x\in (B+\eta_{\alpha_{j_1}}) \cap
(B+\eta_{\alpha_{j_2}})$ if and only if $x+(\eta_{\alpha_{j_1}}+
\eta_{\alpha_{j_2}}) \in (B+\eta_{\alpha_{j_1}}) \cap
(B+\eta_{\alpha_{j_2}})$, so the choice of $G_i(\alpha_{j_1},\alpha_{j_2})$  
is possible by the assumptions on the $\eta_\alpha$'s.)
Suppose that $\ell<\omega$ is such that for any distinct
$j_1,j_2<j$ we have $\eta_{\alpha_{j_1}}\rest \ell \neq
\eta_{\alpha_{j_2}}\rest \ell$ and there are no repetitions in the sequence
$\langle G_i(\alpha_{j_1},\alpha_{j_2})\rest \ell,
G_i(\alpha_{j_2},\alpha_{j_1})\rest \ell:i<\iota\rangle$. Now let 
$u=\{\eta_{\alpha_{j'}}\rest \ell: j'<j\}$, and for $i<\iota$ let
$g_i(\eta_{\alpha_{j_1}}\rest \ell,\eta_{\alpha_{j_2}}\rest
\ell)=G_i(\alpha_{j_1},\alpha_{j_2})\rest \ell$, and let
$h_i(\eta_{\alpha_{j_1}}\rest \ell,\eta_{\alpha_{j_2}}\rest \ell) <\omega$
be such that $G_i(\alpha_{j_1},\alpha_{j_2})\in
\lim\big(T_{h_i(\eta_{\alpha_{j_1}}\rest \ell,\eta_{\alpha_{j_2}}\rest\ell)}
\big)$.  This defines $\bmm=(\ell,u,\bar{h},
\bar{g})\in \Mtk$ and easily $\bbM\models R_\bmm[\alpha_0,\ldots,
\alpha_{j-1}]$.  
\medskip

\noindent (2)\quad Similar to (1). 
\end{proof}
\medskip

The proof of the Proposition is a consequence of the following Claim. 

\begin{claim}
  \label{cl13}
If $\bmm\in\Mtk$ and $\bbM\models R_\bmm[\alpha_0,\ldots, 
  \alpha_{|u_\bmm|-1}]$, then  
\[\rk(\{\alpha_0,\ldots,\alpha_{|u_\bmm|-1}\},\bbM)\leq \ndrk(\bmm)<
  \vare.\]   
\end{claim}

\begin{proof}[Proof of the Claim]
By induction on $\beta$ we show that {\em for every\/}
$\bmm\in\Mtk$ and {\em all\/} distinct $\alpha_0,\ldots,
\alpha_{|u_\bmm|-1}< \lambda$ such that $\bbM\models
R_\bmm[\alpha_0,\ldots, \alpha_{|u_\bmm|-1}]$: 
\begin{quotation}
$\beta\leq \rk(\{\alpha_0,\ldots,\alpha_{|u_\bmm|-1}\}, \bbM)$ implies $\beta\leq  
\ndrk(\bmm)$. 
\end{quotation}
\smallskip

\noindent {\sc Steps $\beta=0$ and $\beta$ is limit:}\quad Straightforward.  
\smallskip

\noindent {\sc Step $\beta=\gamma+1$:}\quad Suppose $\bmm\in\Mtk$ and  
$\alpha_0,\ldots,\alpha_{|u_\bmm|-1}<\lambda$ are such that
$\bbM\models R_\bmm[\alpha_0,\ldots, \alpha_{|u_\bmm|-1}]$ and
$\gamma+1\leq \rk(\{\alpha_0,\ldots,\alpha_{|u_\bmm|-1}\}, \bbM)$. Let
$\nu\in u_\bmm$, so $\nu=\eta_{\alpha_j}\rest \ell_\bmm$ for some
$j<|u_\bmm|$. Since $\gamma+1\leq \rk(\{\alpha_0,\ldots,\alpha_{|u_\bmm|-1}
\}, \bbM)$ we may find $\alpha^*\in \lambda\setminus\{\alpha_0,\ldots,
\alpha_{|u_\bmm|-1}\}$ such that $\bbM\models R_\bmm[\alpha_0,\ldots,
\alpha_{j-1},\alpha^*, \alpha_{j+1}, \ldots,\alpha_{|u_{\bmm}|-1}]$ and
$\rk(\{\alpha_0,\ldots,\alpha_{|u_{\bmm}|-1}, \alpha^*\}, \bbM)\geq
\gamma$. Taking sufficiently large $\ell$ we may use Claim \ref{cl12}(2) to
find $\bn\in\Mtk$ such that $\bmm\sqsubseteq \bn$, $\ell_\bn=\ell$,
$|u_\bn|=|u_\bmm|+1$ and  $\bbM\models
R_\bn[\alpha_0,\ldots,\alpha_{|u_\bmm|-1},\alpha^*]$ and 
$|\{\eta\in u_\bn:\nu\vtl\eta\}|\geq 2$. By the inductive hypothesis we have also
$\gamma\leq \ndrk(\bn)$.  Now we may easily conclude that $\gamma+1\leq
\ndrk(\bmm)$. 
\end{proof}

By our assumptions, $\ndrk(\bmm)<\vare$ for all $\bmm\in\Mtk$. Hence, by
Claims \ref{cl13}+\ref{cl12}(1), we see that $\rk(w,\bbM)<\vare$ for all
$w\in[\lambda]^{<\omega}$, $|w|\geq 2$. Consequently, if $\emptyset\neq w\in
[\lambda]^{<\omega}$ then $\rk(w,\bbM)\leq \vare$. So it should be clear
that the model $\bbM$ witnesses ${\rm NPr}^{1+\vare}(\lambda)$.
\end{proof}

\begin{definition}
  \label{almostDef}
Assume $\bmm,\bn\in\Mtk$. 
\begin{enumerate}
\item We say that {\em $\bmm$, $\bn$ are  essentially the same\/}
  ($\bmm\doteqdot \bn$ in short) if and only if: 
  \begin{itemize}
\item $\ell_\bmm=\ell_\bn$, $u_\bmm=u_\bn$ and 
\item for each $(\eta,\nu)\in (u_\bmm)^{\langle 2\rangle}$ we have\\
--- if $\iota>2$ then   
\[\big\{\{g_i^\bmm(\eta,\nu), g_i^\bmm(\nu,\eta)\}:i<\iota\big\} 
=\big\{\{g^\bn_i(\eta,\nu), g^\bn_i(\nu,\eta)\}:i<\iota\big\},\]
--- if $\iota=2$ then
\[\big\{g_0^\bmm(\eta,\nu), g_0^\bmm(\nu,\eta), g_1^\bmm(\eta,\nu),
  g_1^\bmm(\nu,\eta) \big\} = \big\{g_0^\bn(\eta,\nu), g_0^\bn(\nu,\eta),
  g_1^\bn(\eta,\nu),  g_1^\bn(\nu,\eta) \big\},\]
and for $i,j<\iota$:\\
if $g_i^\bmm(\eta,\nu)= g^\bn_j(\eta,\nu)$, then $h_i^\bmm(\eta,\nu)  
=h^\bn_j(\eta,\nu)$,\\
if $g_i^\bmm(\eta,\nu)= g^\bn_j(\nu,\eta)$, then
$h_i^\bmm(\eta,\nu)=h^\bn_j(\nu,\eta)$. 
\end{itemize}
\item We say that {\em $\bn$ essentially extends 
    $\bmm$\/} ($\bmm\sqsubseteq^* \bn$ in short) if and only if: 
  \begin{itemize}
\item $\ell_\bmm\leq \ell_\bn$, $u_\bmm=\{\eta\rest\ell_\bmm:\eta\in
  u_\bn\}$,   and 
\item for every $(\eta,\nu)\in (u_\bn)^{\langle 2\rangle }$ such that
  $\eta\rest \ell_\bmm \neq \nu\rest\ell_\bmm$ we have
  
--- if $\iota>2$ then   
\[\big\{\{g_i^\bmm(\eta\rest \ell_\bmm,\nu\rest \ell_\bmm),
  g_i^\bmm(\nu\rest \ell_\bmm,\eta\rest \ell_\bmm)\}:i<\iota\big\}  
=\big\{\{g^\bn_i(\eta,\nu)\rest \ell_\bmm, g^\bn_i(\nu,\eta)\rest
\ell_\bmm\}:i<\iota\big\},\] 
--- if $\iota=2$ then
\[\begin{array}{l}
  \big\{g_0^\bmm(\eta\rest \ell_\bmm,\nu\rest \ell_\bmm), g_0^\bmm(\nu\rest
  \ell_\bmm, \eta\rest \ell_\bmm), g_1^\bmm(\eta\rest \ell_\bmm,\nu\rest
  \ell_\bmm),  g_1^\bmm(\nu\rest \ell_\bmm,\eta\rest \ell_\bmm) \big\} =\\
  \big\{g_0^\bn(\eta,\nu)\rest \ell_\bmm, g_0^\bn(\nu,\eta)\rest \ell_\bmm, 
  g_1^\bn(\eta,\nu)\rest \ell_\bmm,  g_1^\bn(\nu,\eta)\rest \ell_\bmm
    \big\},  \end{array}\]
and for $i,j<\iota$:\\
if $g_i^\bmm(\eta\rest\ell_\bmm,\nu\rest \ell_\bmm)= g^\bn_j(\eta,\nu)\rest
\ell_\bmm$, then $h_i^\bmm(\eta\rest\ell_\bmm,\nu\rest \ell_\bmm)
=h^\bn_j(\eta,\nu)$,\\
if $g_i^\bmm(\eta\rest\ell_\bmm,\nu\rest \ell_\bmm)= g^\bn_j(\nu,\eta)\rest
\ell_\bmm$, then $h_i^\bmm(\eta\rest\ell_\bmm,\nu\rest
\ell_\bmm)=h^\bn_j(\nu, \eta)$. 
\end{itemize}
\end{enumerate}
\end{definition}

The reader may wonder why the case of $\iota=2$ is singled out in Definition
\ref{almostDef}. The reason lies in the fact that if $a+b=c+d$ (in $\can$)
then also $a+c=b+d$, Consequently any arrangement of elements of
\[\big\{g_0^\bmm(\eta,\nu), g_0^\bmm(\nu,\eta), g_1^\bmm(\eta,\nu),
  g_1^\bmm(\nu,\eta) \big\}\]
in pairs can be used to produce an object from $\Mtk$. If $\iota>3$ then the
only re-arrangements of this sort are interchanging $g_i^\bmm(\eta,\nu)$ and
$g_i^\bmm(\nu,\eta)$ and/or re-enumerating $\big\{\{g_i^\bmm(\eta,\nu),
g_i^\bmm(\nu,\eta)\}: i<\iota\big\}$; cf the proof of Claim \ref{cl7}.

\section{Cute $\YZR$ and forcing nondisjoint translations} 
In this section we give a property of $\bar{T}$ allowing us to force many
(but not too many) overlapping translations of the corresponding
$\Sigma^0_2$ set. Conditions in the forcing notions come from finite
approximations ({\em bricks\/}) suitably placed on finite subsets of
$\lambda$. An amalgamation property, cute $\YZR$ systems and the splitting
rank on $\lambda$ will all help with the ccc of the forcing notion. 

\begin{definition}
  \label{sysdef}
Let $0<\vare<\omega_1$. {\em A $\YZR(\vare)$--system\footnote{$\YZR$ are  
    the initials of the first author's daughter -- she really wanted to be in
    this paper}\/}  is a tuple $s=(X^s,\bar{r}^s,\bar{\jmath}^s, \bar{k}^s)
=(X,\bar{r},\bar{\jmath}, \bar{k})$ such that
\begin{enumerate}
\item[$(*)_1$] $X$ is a nonempty set of ordinals, 
\item[$(*)_2$] $\bar{r}:[X]^{<\omega}\setminus\{\emptyset\} \longrightarrow 
  \vare+1$, $\bar{\jmath}: [X]^{<\omega}\setminus \{\emptyset\}
  \longrightarrow \omega$, and  $\bar{k}: [X]^{<\omega}\setminus
  \{\emptyset\} \longrightarrow \omega$,  
\item[$(*)_3$] if $\emptyset\neq u\subseteq w\in [X]^{<\omega}$, then
  $\bar{r}(u)\geq \bar{r}(w)$,
\item[$(*)_4$] $\bar{r}(\{a\})>0$ for all $a\in X$,
\item[$(*)_5$]  if $\emptyset\neq w\in [X]^{<\omega}$,
  $w=\{a_0,\ldots,a_{n-1}\}$ (the increasing enumeration) then
  $\bar{k}(w)<n$ and there is no $b\in X\setminus w$ such that   
\[|w\cap b|=\bar{k}(w)  \quad \mbox{ and }\quad \bar{\jmath}\big( (w
  \setminus \{a_{\bar{k}(w)}\})\cup \{b\}\big) =\bar{\jmath}(w) \quad \mbox{
    and }\quad  \bar{r}\big(w\cup\{b\}\big)=\bar{r}(w).\] 
\end{enumerate}
We say that the system $s$ is {\em finite\/} if the set $X^s$ is finite. 
\end{definition}

\begin{example}
  \label{bases}
With the choices of $\bj,\bk$ and $\rksp$ as described in Definition
\ref{hypo2} (for $\vare$ and $\lambda$ as there), the {\em  finite
  $\YZR(\vare)$--system associated with a set $w\in [\lambda]^{<\omega}$\/}
is  $s(w)=(w,\bar{r},\bar{\jmath},\bar{k})$ defined as follows. First, fix
an enumeration $\{v_i^*:i<i^*\}=\{v\subseteq w:v\neq\emptyset\ \wedge \ 
\rksp(v)=-1\}$. Let $J=\max\big(\bj(v):\emptyset\neq v\subseteq w\big)+1$.   
For $\emptyset\neq v\subseteq w$ we define 
\begin{itemize}
\item $\bar{r}(v)=1+\rk^{\rm sp}(v)$,  and $\bar{k}(v)=\bk(v)$, and 
\item if $\rksp(v)\geq 0$, then $\bar{\jmath}(v)=\bj(v)$, and 
\item $\bar{\jmath}(v^*_i)=J+i$ for $i<i^*$. 
\end{itemize}
(It should be clear that the above conditions define a $\YZR(\vare)$--system
indeed.) 
\end{example}

\begin{definition}
  \label{cutedef}
\begin{enumerate}
\item Assume $q,s$ are $\YZR(\vare)$--systems. A {\em quasi--embedding of
    $q$ in $s$\/} is an increasing injection $\varphi:X^q\longrightarrow
  X^s$ such that for all nonempty finite $v\subseteq X^q$ we have 
  \begin{itemize}
\item $\bar{r}^s(\varphi[v]\big)=\bar{r}^q(v)$ and
    $\bar{k}^s(\varphi[v]\big)=\bar{k}^q(v)$, and 
\item if $\bar{r}^q(v)>0$, then $\bar{\jmath}^s(\varphi[v]\big)=
  \bar{\jmath}^q(v)$. 
 \end{itemize}
\item If $w\subseteq X^s$, then an increasing injection
  $\varphi:w\longrightarrow X^s$ is {\em a quasi--embedding\/} if it is a
  quasi embedding of the (naturally defined) restricted
  $\YZR(\vare)$--system $s\rest w$ into $s$. 
\item A $\YZR(\vare)$--system $S$ is {\em cute\/} if $X^S=\omega$ and  
for every finite $\YZR(\vare)$--system $q$ and an $M<\omega$, there is a
quasi--embedding $\varphi$ of $q$ in $S$ with $\rng(\varphi)\subseteq
[M,\omega)$.  
\end{enumerate}
\end{definition}

\begin{theorem}
  \label{getcute}
For every $0<\vare<\omega_1$ there exists a cute $\YZR(\vare)$--system. 
\end{theorem}

\begin{proof}
Assume $0<\vare<\omega_1$. Let $\cS$ consist of all finite
$\YZR(\vare)$--systems  $s=(N^s,\bar{r}^s,\bar{\jmath}^s, \bar{k}^s)$ such
that $0<N^s<\omega$. For $q,s\in\cS$ we will say that {\em $s$ extends
  $q$\/}, in short $q\preceq s$, if and only if $N^q\leq N^s$,
$\bar{r}^q\subseteq \bar{r}^s$, $\bar{\jmath}^q\subseteq \bar{\jmath}^s$,
and $\bar{k}^q\subseteq \bar{k}^s$.

\begin{claim}
  \label{cl0}
The relation $\preceq$ is a partial order on $\cS$. As a matter of fact,
$(\cS,\preceq)$ is the Cohen forcing notion.
\end{claim}
 
\begin{claim}
  \label{cl1}
  Suppose that $s\in \cS$ and $q=(X^q,\bar{r}^q,\bar{\jmath}^q,\bar{k}^q)$ 
  is a finite $\YZR(\vare)$--system. Then there are $t\succeq s$ and an
  increasing injection $\varphi:X^q\longrightarrow [N^s,N^t)$ such that for
  each nonempty $v\subseteq X^q$ we have  
\[\bar{r}^t(\varphi[v]\big)=\bar{r}^q(v) \mbox{ and }
\bar{\jmath}^t(\varphi[v]\big)=\bar{\jmath}^q(v) \mbox{ and }
\bar{k}^t(\varphi[v]\big)=\bar{k}^q(v).\]
\end{claim}

\begin{proof}[Proof of the Claim]
Without loss of generality, $X^q=N<\omega$.  Let $N^t=N^s+N$ and let
$\varphi:X^q\longrightarrow [N^s,N^t): m\mapsto N^s+m$.  We also let  
\[J_0=\max\big(\rng(\bar{\jmath}^s)\cup\rng(\bar{\jmath}^q)\big)+1\quad
  \mbox{ and }\quad J_1=J_0+(2^{N^s}-1)\cdot (2^N-1)\]
and we fix a bijection 
\[\psi:\big\{u\subseteq N^t: u\cap N^s\neq\emptyset\neq u\cap
  [N^s,N^t)\big\} \longrightarrow [J_0,J_1).\]
Now, to define $\bar{r}^t,\bar{\jmath}^t$ and $\bar{k}^t$ we put for
$u\subseteq N^t$:   
\begin{itemize}
\item if $u\subseteq N^s$, then $\bar{r}^t(u)=\bar{r}^s(u)$,
  $\bar{\jmath}^t(u) =\bar{\jmath}^s(u)$ and $\bar{k}^t(u)=\bar{k}^s(u)$, 
\item  if $u\subseteq [N^s,N^t)$, then  $\bar{r}^t(u)=\bar{r}^q \big(
  \varphi^{-1}[u]\big)$, $\bar{\jmath}^t(u) =\bar{\jmath}^q \big(
  \varphi^{-1}[u]\big)$ and $\bar{k}^t(u)=\bar{k}^q\big(\varphi^{-1}[u]
  \big)$,   
\item if $u\cap N^s\neq \emptyset \neq u\cap [N^s,N^t)$, then
  $\bar{r}^t(u)=0$, $\bar{k}^t(u)=0$ and $\bar{\jmath}^t(u)=\psi(u)$. 
\end{itemize}
This completes the definition of $t=(N^t,\bar{r}^t,\bar{\jmath}^t,
\bar{k}^t)$. To verify that $t\in\cS$ note that clauses $(*)_1$--$(*)_4$ of
Definition \ref{sysdef} follow immediately from our choices. 

Let us argue that \ref{sysdef}$(*)_5$ is satisfied too. Suppose that
$\emptyset\neq u\subseteq N^t$, $|u|=n$ and $u=\{a_0,\ldots,a_{n-1}\}$
is the increasing enumeration. Straightforward from the definitions above, 
$\bar{k}(u)<n$. Now, 
\begin{itemize}
\item if $u\cap N^s\neq\emptyset\neq u\cap [N^s,N^t)$,  then no other
  $u'\subseteq N^t$ satisfies $\bar{\jmath}(u')=\bar{\jmath}(u)$. At the
  same time $(u\setminus \{a\})\cup\{b\}\neq u$ for $a\in u$ and $b\notin
  u$.   
\item If $u\subseteq N^s$, then
  \begin{itemize}
  \item for every $b\in N^s\setminus u$, by $(*)_5$ for $s$, either $|u\cap
    b|\neq\bar{k}^s(u)=\bar{k}^t(u)$ or\\
 $\bar{\jmath}^t(u\setminus \{a_{\bar{k}^t(u)}\}\cup \{b\})=
 \bar{\jmath}^s(u\setminus    \{a_{\bar{k}^s(u)}\}\cup \{b\})
 \neq\bar{\jmath}^s(u)=\bar{\jmath}^t(u)$ or \\
  $\bar{r}^t(u\cup\{b\})=\bar{r}^s(u\cup\{b\})<\bar{r}^s(u)=\bar{r}^t(u)$, 
\item   for every $b\in  [N^s,N^t)$ we have 
  \begin{itemize}
\item  $\bar{r}^t(u\cup\{b\})=0<\bar{r}^t(u)$ when $n=1$ and 
\item  $\bar{\jmath}^t(u \setminus \{a_{\bar{k}^t(u)}\}\cup \{b\})\neq
  \bar{\jmath}^s(u)=\bar{\jmath}^t(u)$ when $n>1$.   
\end{itemize}
\end{itemize}
\item If $u\subseteq [N^s,N^t)$, then
  \begin{itemize}
  \item for every $b\in [N^s,N^t)\setminus u$, by $(*)_5$ for $q$, either\\
    $|u\cap b|=|\varphi^{-1}[u]\cap \varphi^{-1}(b)| \neq \bar{k}^q( 
    \varphi^{-1} [ u])=\bar{k}^t(u)$ or\\
 $\bar{\jmath}^t(u\setminus \{a_{\bar{k}^t(u)}\}\cup \{b\})=
 \bar{\jmath}^q(\varphi^{-1} [u\setminus    \{a_{\bar{k}^t(u)}\}\cup \{b\}])
 \neq\bar{\jmath}^q(\varphi^{-1} [u])=\bar{\jmath}^t(u)$ or \\
  $\bar{r}^t(u\cup\{b\})=\bar{r}^q(\varphi^{-1} [u\cup\{b\}])<
  \bar{r}^q(\varphi^{-1} [u])=\bar{r}^t(u)$,  
\item   for every $b\in N^s$ we have 
  \begin{itemize}
\item  $\bar{r}^t(u\cup\{b\})=0<\bar{r}^t(u)$ when $n=1$ and 
\item  $\bar{\jmath}^t(u \setminus \{a_{\bar{k}^t(u)}\}\cup \{b\})\neq
  \bar{\jmath}^q(\varphi^{-1} [u])=\bar{\jmath}^t(u)$ when $n>1$.   
\end{itemize}
\end{itemize}
\end{itemize}
Consequently, in any possible case there is no $b\in N^t\setminus u$ such
that    
\[|u\cap b|=\bar{k}^t(u)  \quad \mbox{ and }\quad \bar{\jmath}^t (u\setminus   
  \{a_{\bar{k}^t (u)}\}\cup \{b\}) =\bar{\jmath}^t (u) \quad \mbox{ and
  }\quad \bar{r}^t (u\cup\{b\})=\bar{r}^t (u).\] 
Therefore, $q\in\cS$ and easily it is as required.
\end{proof}

Let $\langle q_i:i<\omega\rangle$ list with infinite repetitions all
elements of $\cS$. Use Claim \ref{cl1} to construct a sequence $\langle   
s_i:i<\omega\rangle$ such that for all $i<\omega$:
\begin{itemize}
\item $s_i\in\cS$, $s_i\preceq s_{i+1}$,
\item for some increasing injection $\varphi_i:N^{q_i}\longrightarrow
  [N^{s_i},N^{s_{i+1}})$ we have 
\[\bar{r}^{s_{i+1}}(\varphi_i[v]\big)=\bar{r}^{q_i}(v) \mbox{ and }
\bar{\jmath}^{s_{i+1}} (\varphi[v]\big)=\bar{\jmath}^{q_i}(v) \mbox{ and } 
\bar{k}^{s_{i+1}} (\varphi[v]\big)=\bar{k}^{q_i}(v)\]
for all $\emptyset\neq v\subseteq N^{q_i}$.
\end{itemize}
Then let $S=(\omega,\bar{r}^S,\bar{\jmath}^S,\bar{k}^S)$ be defined by 
\[\bar{r}^S=\bigcup_{i<\omega}\bar{r}^{s_i},\quad 
\bar{\jmath}^S=\bigcup_{i<\omega}\bar{\jmath}^{s_i},\quad 
\bar{k}^S=\bigcup_{i<\omega}\bar{k}^{s_i}.\]
Plainly, $S$ is a cute $\YZR(\vare)$--system. 
\end{proof}

\begin{hypothesis}
  \label{hypo3}
In the rest of this section we assume that 
\begin{itemize}
\item $2\leq\iota<\omega$, and  $\bar{c}=\langle c_m:m<\omega
  \rangle\subseteq \omega$,   
\item $T_m\subseteq {}^{\omega>} 2$ (for $m<\omega$) are trees with no 
  maximal nodes, $\bar{T}=\langle T_m:m<\omega\rangle$, and
  $B=\bigcup\limits_{m<\omega} \lim(T_m)$, 
\item there are pairwise different $\rho_0,\rho_1,\rho_2\in \can$ such that   
  \[\big|\big(\rho_j+B\big)\cap \big(\rho_{j'}+B\big)\big|\geq 2\iota\] 
for $j,j'<3$,
\item $\Mtk$ is defined as in Definition \ref{mtkDef}  and 
\item $S=(\omega,\bar{r},\bar{\jmath},\bar{k})$ is a cute
  $\YZR(\vare)$--system, $0<\vare<\omega_1$.  
\end{itemize}
\end{hypothesis}

\begin{definition}
\label{brick}
\begin{enumerate}
\item An {\em $(S,\iota,\bar{T},\bar{c})$--brick\/} is a tuple 
\[\gb=(w^\gb,n^\gb,\bar{\eta}^\gb,\bar{h}^\gb,\bar{g}^\gb, \cM^\gb)
=(w,n,\bar{\eta},\bar{h},\bar{g},\cM)\]
such that 
\begin{enumerate}
\item[$(\boxplus)_1$] $w\in [\omega]^{<\omega}$, $|w|\geq 3$, $0<n<\omega$. 
\item[$(\boxplus)_2$] $\bar{\eta}=\langle \eta_a :a \in w\rangle$ is a
  sequence of linearly independent vectors in ${}^n 2$ (over the field
  $\bbZ_2$); so in particular $\eta_a \in {}^n2$ are pairwise
  distinct non-zero sequences (for $a \in w$).  
\item[$(\boxplus)_3$] $\bar{h}=\langle h_i:i<\iota\rangle$, where
  $h_i:w^{\langle 2\rangle}\longrightarrow \omega$, and
  $c_{h_i(a ,b )}\leq n$ for $(a,b)\in w^{\langle 2\rangle}$ and $i<\iota$,
  and  $\bar{g} =\langle g_i:i<\iota\rangle$, where $g_i:w^{\langle
    2\rangle} \longrightarrow \bigcup\limits_{m<\omega} (T_m\cap {}^n2)$ for
  $i<\iota$. 
\item[$(\boxplus)_4$] Letting $n^*=n$, $u^*=\{\eta_a :a \in w\}$,
  $h^*_i(\eta_a , \eta_b )= h_i(a ,b )$ and
  $g^*_i(\eta_a , \eta_b )= g_i(a ,b )$ we have
  $(n^*,u^*, \bar{h}^*,\bar{g}^*)\in \Mtk$.
\item[$(\boxplus)_5$] $\cM$ consists of all $\bmm\in\Mtk$ such that for some 
  $\ell_*,w_*$ we  have   
\begin{enumerate}
\item[$(\boxplus)_5^{\rm a}$] $w_*\subseteq w$, $3\leq |w_*|$,
  $0<\ell_\bmm=\ell_*\leq n$,   and for each $(a ,b )\in
  (w_*)^{\langle 2\rangle}$ and $i<\iota$ we have
  $c_{h_i(a ,b )}\leq \ell_*$,
\item[$(\boxplus)_5^{\rm b}$] $u_\bmm=\{\eta_a \rest \ell_*: a \in
  w_*\}$ and $\eta_a \rest\ell_*\neq \eta_b \rest\ell_*$ for distinct
  $a ,b \in w_*$,
\item[$(\boxplus)_5^{\rm c}$] $\bar{h}_\bmm=\langle
  h^\bmm_i:i<\iota\rangle$, where
  \[h^\bmm_i:(u_\bmm)^{\langle 2\rangle} \longrightarrow \omega: (\eta_a
    \rest  \ell_*,\eta_b \rest \ell_*)\mapsto h_i(a ,b ),\]
\item[$(\boxplus)_5^{\rm d}$] $\bar{g}_\bmm=\langle
  g^\bmm_i:i<\iota\rangle$, where  
\[g^\bmm_i:(u_\bmm)^{\langle 2\rangle} \longrightarrow \bigcup\limits_{m<
    \omega} (T_m\cap {}^{\ell_*} 2):(\eta_a \rest \ell_*,\eta_b \rest 
  \ell_*)\mapsto g_i(a ,b )\rest \ell_*\]
\end{enumerate}
In the above situation we will write $\bmm= \bmm(\ell_*,w_*)=
\bmm^\gb(\ell_*,w_*)$. 
\item[$(\boxplus)_6$] If $\bmm(\ell,w_0), \bmm(\ell,w_1)\in\cM$ , $\rho\in {}^\ell
  2$ and $\bmm(\ell,w_0)\doteqdot\bmm(\ell,w_1)+\rho$, then the order
  isomorphism $\pi:w_0\longrightarrow w_1$ is a quasi--embedding  and 
 $(\eta_a \rest \ell)+\rho=\eta_{\pi(a )}\rest\ell$ for all
 $a \in w_0$.    
\item[$(\boxplus)_7$] If $\bmm(\ell_*,w_*)\in \cM$, $a \in w_*$,
  $|a \cap w_*|= \bar{k}(w_*)$, $\bar{r}(w_*)=0$, and $\bmm(\ell_*,w_*) 
  \sqsubseteq^* \bn\in \cM$,  then   $|\{\nu\in
  u_\bn:(\eta_a \rest\ell_*) \trianglelefteq \nu\}|=1$.    
\end{enumerate}
\item Suppose that $t_m=T_m\cap  {}^{n\geq} 2$ and $c_m\leq n$ for
  $m<M<\omega$. Let $\bar{t}=\langle t_m:m<M\rangle$ and
  $\bar{d}=\bar{c}\rest M$. An $(S,\iota,\bar{T},\bar{c})$--brick $\gb$ such
  that $n^\gb=n$,  $h_i^\gb(a ,b )<M$ for all $(a ,b )\in
  \big(w^\gb\big)^{\langle 2\rangle}$ and $i<\iota$ will be also
  called an {\em $(S,\iota,\bar{t},\bar{d})$--brick}.  
\item For bricks $\gb_0,\gb_1$ we write $\gb_0\Subset \gb_1$\quad if
  and only if   
\begin{itemize}
\item $w^{\gb_0}\subseteq w^{\gb_1}$, $n^{\gb_0}\leq n^{\gb_1}$, and 
\item $\eta^{\gb_0}_a \trianglelefteq \eta^{\gb_1}_a $ for all
  $a \in w^{\gb_0}$,   and    
\item $h^{\gb_1}_i\rest (w^{\gb_0})^{\langle 2\rangle}= h^{\gb_0}_i$ and
  $g^{\gb_0}_i(a ,b ) \trianglelefteq g^{\gb_1}_i(a ,b )$ for
  $i<\iota$ and $(a ,b )\in  (w^{\gb_0})^{\langle 2\rangle}$.
\end{itemize}
\end{enumerate}
\end{definition}

\begin{remark}
  \label{rembri}
  \begin{enumerate}
\item
  Note that in $(\boxplus)_5$ of Definition \ref{brick}, the set $w_*$ is 
  notdetermined uniquely by $\bmm$ and we may have
  $\bmm^\gb(\ell,w_0)=\bmm^\gb(\ell,w_1)$ for distinct  $w_0,w_1\subseteq
  w$.  
\item If $w_*\subseteq w^\gb$ has at least $3$ elements, then
    $\bmm^\gb(n^\gb,w_*)\in \cM^\gb$. 
\item We will use $(S,\iota,\bar{t},\bar{d})$--bricks for $\bar{t}=\langle
  t_m: m<M\rangle$ and $\bar{d}=\langle d_m:m<M\rangle$ (see Definition 
  \ref{brick}(2)) even if full $\bar{T},\bar{c}$ are not defined. In these
  cases we mean {\em for $\bar{T}=\langle T_m:m<\omega\rangle$ where for
    $m<M$ we  have 
\[T_m=\Big\{\nu\in {}^{\omega>} 2:\nu\rest n\in t_m\ \wedge\
  \big(\forall k<\lh(\nu)\big)\big(n\leq k\ \Rightarrow\ \nu(k)=0\big) 
  \Big\}\] 
and   $T_m={}^{\omega>} 2$ when $M\leq m<\omega$, and  some
$\bar{c}$ such that $c_m=d_m\leq n$ whenever $m<M$.}  (See Definition
\ref{fmtkDef}.) 
  \end{enumerate}
\end{remark}

\begin{observation}
\label{obsbrick}
Assume $\gb$ is an $(S,\iota,\bar{T},\bar{c})$--brick. Then:
\begin{enumerate}
\item $n^\gb\geq |w^\gb|$.
\item If $w^*\subseteq w^\gb$ and $|w^*|\geq 3$  then there is a unique
  $(S,\iota,\bar{T},\bar{c})$--brick $\gb^*$ such that $w^{\gb^*}=w^*$,
  $n^{\gb^*}=n^\gb$  and $\gb^*\Subset \gb$.\\
 We may write $\gb^*=\gb\rest  w^*$ then. 
\item If $\bmm=(n^*,u^*,\bar{h}^*,\bar{g}^*)$ is as given by Definition 
\ref{brick}$(\boxplus)_4$, then $\bmm=\bmm^\gb(n^\gb,w^\gb)\in\cM^\gb$.   
\item If $w_0\subseteq w\subseteq w^\gb$, $\bmm^\gb(\ell,w)\in\cM^\gb$ and
  $3\leq |w_0|$, then $\bmm^\gb(\ell,w_0)\in\cM^\gb$.
\item If $\varphi:w^\gb\longrightarrow \omega$ is a quasi--embedding (into
  $S$) then there is a unique $(S,\iota,\bar{T},\bar{c})$--brick $\gb^*$
  such that  
\begin{itemize}
\item $w^{\gb^*}=\varphi[w^\gb]$, $n^{\gb^*}=n^{\gb}$, $\cM^{\gb^*}=
    \cM^{\gb}$, and 
\item $\eta^\gb_a =\eta^{\gb^*}_{\varphi(a )}$, $h^\gb_i(a ,
  b )= h^{\gb^*}_i(\varphi(a ), \varphi(b ))$ and $g^\gb_i(a , 
  b )= g^{\gb^*}_i(\varphi(a ), \varphi(b ))$ for all relevant
  $a ,b ,i$. 
  \end{itemize}
This $\gb^*$ will be denoted $\varphi(\gb)$.
\end{enumerate}
\end{observation}

\begin{definition}
\label{defamal}
We say that $\bar{T}$ has {\em $(\bar{c},S)$--controlled amalgamation
  property\/} if there is a sequence $\bar{\gb}=\langle \gb_n:n<\omega
\rangle$ of $(S,\iota,\bar{T},\bar{c})$--bricks such that
\begin{enumerate}
\item $\gb_n\Subset \gb_{n+1}$ for each $n<\omega$,
\item $\bigcup\limits_{n<\omega} w^{\gb_n}=\omega$ and
  $\lim\limits_{n\to\infty} n^{\gb_n}=\infty$,
\item IF 
\begin{enumerate}
\item[(a)] $n<\omega$, $u\subseteq w\subseteq w^{\gb_n}$, $3\leq |w|$,  
\item[(b)] $\bar{k}\big(v\cup\{\delta\}\big)\neq |\delta\cap 
  v|$ whenever $v\subseteq u$ and $\delta\in w\setminus u$ and 
  $\bar{r}\big(v\cup\{\delta\}\big)=0$, 
\item[(c)]   $\pi_0,\pi_1:w \longrightarrow \omega$ are quasi--embeddings 
  (into $S$) such that $\pi_0(a )=\pi_1(a )$  for $a \in u$ and
  $\pi_0[w\setminus u]\cap \pi_1[w\setminus u]=\emptyset$, 
\end{enumerate}
THEN there is a $K<\omega$ and a quasi--embedding
$\pi:\rng(\pi_0)\cup\rng(\pi_1) \longrightarrow w^{\gb_K}$ (into $S$) such
that 
\[(\pi\circ\pi_0)(\gb_n\rest w)\Subset \gb_K\rest
  (\pi\circ\pi_0[w])\quad\mbox{ and }\quad
  (\pi\circ\pi_1)(\gb_n\rest w)\Subset \gb_K\rest (\pi\circ\pi_1[w]),\] 
\item IF 
\begin{enumerate}
\item[(a)] $n<\omega$, $w\subseteq w^{\gb_n}$, $3\leq |w|$,  
\item[(b)] $\pi_0:w \longrightarrow \omega$ is a quasi--embedding 
  (into $S$) and $\rng(\pi_0)\subseteq u\in [\omega]^{<\omega}$, 
\end{enumerate}
THEN there is a $K<\omega$ and a quasi--embedding
$\pi:u\longrightarrow w^{\gb_K}$ (into $S$) such that 
$(\pi\circ\pi_0)(\gb_n\rest w)\Subset \gb_K\rest
  (\pi\circ\pi_0[w])$.
\end{enumerate}
\end{definition}

The name of the $(\bar{c},S)$--controlled {\em amalgamation property\/}
comes from the third part of the demand. This demand is taylored to
guarantee that {\bf if}
\begin{enumerate}
\item[(i)] $\gb_0,\gb_1$ are $(S,\iota,\bar{T},\bar{c})$--bricks,
  $n^{\gb_0}=n^{\gb_1}$, $|w^{\gb_0}|=|w^{\gb_1}|$  and  
\item[(ii)]  the order isomorphism $\pi:w^{\gb_0} \longrightarrow w^{\gb_1}$
is a quasi--embedding (into $S$), $\pi(a )=a $  for $a \in
w^{\gb_0}\cap  w^{\gb_1}$,  and $\pi(\gb_0)=\gb_1$, and
\item[(iii)] for each $v\subseteq w^{\gb_0}\cap w^{\gb_1}$ and $\delta\in
  w^{\gb_0} \setminus w^{\gb_1}$,
\[\bar{r}\big(v\cup\{\delta\}\big)=0\quad\Rightarrow \quad \bar{k}\big(
  v\cup\{\delta\}\big) \neq |\delta\cap v|,\]
\end{enumerate}
{\bf then} there is an $(S,\iota,\bar{T},\bar{c})$--brick $\gb$ and a quasi
embedding $\pi^*:w^{\gb_0}\cup w^{\gb_1}\longrightarrow \omega$ such that 
$\pi^*(\gb_0)\Subset \gb$ and $\pi^*(\gb_1)\Subset \gb$. Such $\gb$ may be
thought of as an amalgamation of $\gb_0,\gb_1$ over $w^{\gb_0}\cap
w^{\gb_1}$. Note that we may demand the existence of amalgamations only when
$\gb_0,\gb_1$ satisfy condition (iii). Without it the intended $\gb$ would
have to violate demand \ref{brick}(1)$(\boxplus)_7$.
\medskip

In the next section we will construct $\bar{T}$ with the
$(\bar{c},S)$--controlled amalgamation property. Here we show the main
reason to consider such $\bar{T}$ and the associated $\Sigma^0_2$ sets.

\begin{theorem}
 \label{force}
Assume that 
\begin{enumerate}
\item $2\leq\iota<\omega$,  and 
  $\bar{c}=\langle c_m:m<\omega\rangle\subseteq \omega$,  
\item $T_m\subseteq {}^{\omega>} 2$ (for $m<\omega$) are trees with no 
  maximal nodes, $\bar{T}=\langle T_m:m<\omega\rangle$, and
  $B=\bigcup\limits_{m<\omega} \lim(T_m)$, 
\item $S=(\omega,\bar{r},\bar{\jmath},\bar{k})$ is a cute 
  $\YZR(\vare)$--system, $0<\vare<\omega_1$,
\item $\bar{T}$ has $(\bar{c},S)$--controlled amalgamation
  property , and
\item ${\rm NPr}^\vare(\lambda)$ holds true. 
\end{enumerate}
Then there is a ccc forcing notion $\bbP$ of size $\lambda$ such that    
\[\begin{array}{l}
\forces_{\bbP}\mbox{`` there is a sequence
    }\langle\eta_\alpha:\alpha<\lambda\rangle
\mbox{ of distinct elements of $\can$ such that}\\
\qquad \big|(\eta_\alpha+B)\cap
(\eta_\beta+B)\big|\geq 2\iota\mbox{ for all }\alpha,\beta<\lambda\mbox{
    ''.} 
\end{array}\]
\end{theorem}

\begin{proof}
Let a sequence $\bar{\gb}=\langle \gb_n:n<\omega
\rangle$ of $(S,\iota,\bar{T},\bar{c})$--bricks witness the 
$(\bar{c},S)$--controlled amalgamation property for $\bar{T}$. 

First, for $a,b\in \omega$ and $i<\iota$ let
\begin{itemize}
\item $\eta^*_a=\bigcup\big\{\eta^{\gb_n}_a:a\in w^{gb_n},\ n\in
  \omega\big\}$,
\item $g^*_i(a,b)=\bigcup\big\{g^{\gb_n}(a,b): a,b\in w^{gb_n},\ n\in\omega
  \big\}$,
\item $h^*_i(a,b)=h^{\gb_n}(a,b)$ for some (equivalently: all) $n\in\omega$
  such that $a,b\in w^{\gb_n}$.
\end{itemize}
Then $\eta^*_a\in\can$, $g^*_i(a,b)\in\lim\big(T_{h^*_i(a,b)}\big)$, and
$\eta^*_a+g^*_i(a,b)= \eta^*_b+g^*_i(b,a)$, and there are no repetitions in
$\langle g^*_i(a,b),g_i^*(b,a):i<\iota\rangle$ nor in $\langle \eta^*_a:a\in
\omega\rangle$. Hence $\langle \eta^*_a:a\in\omega\rangle$ is a sequence of
distinct elements of $\can$ such that
\[\big|(\eta^*_a+B)\cap (\eta^*_b+B)\big|\geq 2\iota\]
for all $a,b\in\omega$. So we are done when $\lambda=\omega$. 

Now, let us assume that $\lambda$ is uncountable and let $\bbM= \big(\lambda,
\{R^\bbM_{n,j}\}_{n,j<\omega}\big)$ be the model fixed in Definition
\ref{hypo2}, let $\rksp$ be the associated rank and let $\bj,\bk:
[\lambda]^{<\omega}\setminus\{\emptyset\} \longrightarrow \omega$ 
be the ``witness functions'' fixed there.

A condition in $\bbP$ is a tuple 
\[p=(u^p,n^p,\bar{\eta}^p,\bar{h}^p,\bar{g}^p)\]
such that $u^p\in [\lambda]^{<\omega}$, $3\leq |u^p|$ and for some
quasi--embedding $\varphi:u^p \longrightarrow \omega$ of the
system $s(u^p)$ associated with $u^p$ (see Definition \ref{bases}) into $S$
and for some $N<\omega$ we have
\begin{itemize}
\item $\varphi[u^p] \subseteq w^{\gb_N}$, $n^p\leq n^{\gb_N}$,
  $\bmm^{\gb_N}(n^p, \varphi^p[u]) \in \cM^{\gb_N}$, 
\item $\bar{\eta}^p=\langle \eta^p_\alpha:\alpha\in u^p\rangle$ and 
  $\eta^p_\alpha= \eta^{\gb_N}_{\varphi(\alpha)}\rest n^p$,
\item $\bar{h}^p=\langle h^p_i:i<\iota\rangle$, where
\item  $h^p_i:\big(u^p\big)^{\langle 2\rangle}\longrightarrow \omega$ are such that 
    $h^p_i(\alpha,\beta)=h^{\gb_N}_i(\varphi(\alpha),\varphi(\beta))$,
\item $\bar{g}^p=\langle g^p_i:i<\iota\rangle$, where 
  $g^p_i:\big(u^p\big)^{\langle 2\rangle}\longrightarrow
    \bigcup\limits_{m<\omega} (T_m\cap {}^{n^p}2)$ are such that  
    $g^p_i(\alpha,\beta)=g^{\gb_N}_i(\varphi(\alpha),\varphi(\beta))\rest n_p$.
\end{itemize}
(For $\varphi$ and $N$ as above we say that they {\em witness $p\in\bbP$.}\/) 

A condition $q\in\bbP$ {\em is stronger than \/} $p\in
\bbP$ ($p\leq q$ in short) if and only if 
\begin{itemize}
\item $u^p\subseteq u^q$, $n^p\leq n^q$, and 
\item $\eta^p_\alpha\trianglelefteq \eta^q_\alpha$ for all
  $\alpha\in u^p$,   and    
\item $h^q_i\rest (u^p)^{\langle 2\rangle}= h^p_i$ and
  $g^p_i(\alpha,\beta) \trianglelefteq g^q_i(\alpha,\beta)$ for 
  $i<\iota$ and $(\alpha,\beta)\in  (u^p)^{\langle 2\rangle}$.
\end{itemize}
Clearly, $(\bbP,\leq)$ is a partial order of size $\lambda$.

\begin{claim}
\label{cl3}
\begin{enumerate}
\item Suppose that $u\subseteq \lambda$ is a finite set with at least 3
  elements  and $\varphi:u\longrightarrow \omega$ is a quasi--embedding of
  $s(u)$ into  $S$. Assume $\varphi[u]\subseteq w^{\gb_N}$ and $n\leq
  n^{\gb_N}$ is such that $\bmm^{\gb_N}(n,\varphi[u])\in \cM^{\gb_N}$. 
  Then there is a unique condition $p=p(n,\varphi,N) \in \bbP$ such that
  $n^p=n$ and $\varphi$ and $N$ witness $p\in\bbP$.  
\item Assume that $\emptyset\neq u_0\subseteq u_1\subseteq \lambda$, $u_1$
  finite, and $\varphi: u_1\longrightarrow \omega$ is a quasi--embedding
  into   $S$. Suppose $n_0,n_1,K_0,K_1$ are such that $p(n_0,\varphi\rest
  u_0,K_0)$ and $p(n_1,\varphi,K_1)$ are well defined and $n_0\leq n_1$,
  $K_0\leq K_1$. Then $p(n_0,\varphi\rest u_0,K_0) \leq p(n_1,\varphi, K_1)$.    
\end{enumerate}
\end{claim}

\begin{claim}
  \label{cl4}
$\bbP$ has the Knaster property. 
\end{claim}

\begin{proof}[Proof of the Claim]
  Suppose that $\langle p_\xi:\xi<\omega_1\rangle$ is a sequence of pairwise 
distinct conditions from $\bbP$ and let
\[p_\xi=\big(u^\xi,n^\xi,\bar{\eta}^\xi,\bar{h}^\xi,\bar{g}^\xi\big)\]
where $\bar{\eta}^\xi=\langle\eta^\xi_\alpha: \alpha\in u^\xi\rangle$,
$\bar{h}^\xi= \langle h^\xi_i:i<\iota\rangle$, and $\bar{g}^\xi= \langle
g^\xi_i:i<\iota\rangle$. Let $\varphi_\xi$ and $N_\xi$ witness
$p_\xi\in\bbP$. 

Use the standard $\Delta$--system cleaning procedure to find an uncountable
set $A\subseteq \omega_1$ such that the following demands
$(\oplus)_1$--$(\oplus)_4$ are satisfied.
\begin{enumerate}
\item[$(\oplus)_1$] $\{u^\xi:\xi\in A\}$ forms a $\Delta$--system with
  kernel $u$.
\item[$(\oplus)_2$] If  $\xi,\varsigma\in A$, then $|u^\xi|=
  |u^\varsigma|$ and $n^\xi=n^\varsigma$. 
\item[$(\oplus)_3$] If $\xi<\varsigma$ are from $A$ and $\pi: u^\xi
  \longrightarrow u^\varsigma$ is the order isomorphism, then  
  \begin{enumerate}
  \item[(a)] $\pi(\alpha)=\alpha$ for $\alpha\in u^\xi\cap
    u^\varsigma$,  
  \item[(b)] if $\emptyset\neq v\subseteq u^\xi$, then
    $\rksp(v)=\rksp(\pi[v])$, $\bj(v)= \bj(\pi[v])$ and
    $\bk(v)=\bk(\pi[v])$,  
  \item[(c)] $\eta_\alpha^\xi=\eta_{\pi(\alpha)}^{\varsigma}$ (for
    $\alpha\in  w_\xi$), 
  \item[(d)] $g_i(\alpha,\beta)=g_i(\pi(\alpha),\pi(\beta))$ and
    $h_i(\alpha,\beta) =h_i(\pi(\alpha),\pi(\beta))$ for $(\alpha,\beta)\in
    (w_\xi)^{\langle 2\rangle}$ and $i<\iota$,
  \end{enumerate}
and 
\item[$(\oplus)_4$] $\rng(\varphi_\xi)=\rng(\varphi_\varsigma)=w$  and
  $N_\xi= N_\varsigma=N$ for $\xi,\varsigma\in A$.
\end{enumerate}
Note that then also 
\begin{enumerate}
\item[$(\oplus)_5$] if $\xi\in A$, $v\subseteq u$ and $\delta\in
  u^\xi\setminus u$ are such that $\rksp\big(v\cup\{\delta\}\big)=-1$, then
  $\bk\big(v\cup\{\delta\}\big)\neq |\delta\cap v|$.
\end{enumerate}
[Why? Suppose $\rksp\big(v\cup\{\delta\}\big)=-1$ and
$k=\bk\big(v\cup\{\delta\}\big)= |\delta\cap v|$, $j=\bj\big( v\cup
\{\delta\}\big)$.  For $\varsigma\in A$ let  $\pi_\varsigma: u^\xi
\longrightarrow u^\varsigma$ be the order isomorphism and let
$\delta_\varsigma=\pi_\varsigma(\delta)$. By $(\oplus)_3$ we know that
$k=\bk\big(v\cup\{\delta_\varsigma\}\big) =|\delta_\varsigma\cap v|$ and
$j=\bj\big(v\cup\{\delta_\varsigma\}\big)$. Therefore, letting
$v\cup\{\delta\}=\{a_0,\ldots,a_{n-1}\}$ be the increasing enumeration, for
every $\varsigma\in A$ we have $\bbM\models R_{n,j}[a_0,\ldots,a_{k-1},
\delta_\varsigma, a_{k+1},\ldots, a_{n-1}]$. Hence the set 
\[\{b<\lambda: \bbM\models R_{n,j}[a_0,\ldots,a_{k_1}, b , a_{k+1},\ldots,
  a_{n-1}] \}\]
is uncountable, contradicting $(\circledast)_{\rm f}$ of \ref{hypo2}.]
\medskip

Let us argue that for distinct $\xi,\varsigma$ from $A$ the conditions
$p_\xi,p_\varsigma$ are compatible. So let $\xi,\varsigma\in A$,
$\xi<\varsigma$. Let $v^*=u^\xi\cup u^\varsigma$ and let $s(v^*)$ be the
finite $\YZR(\vare)$--system associated with $v^*$ (see Definition
\ref{bases}). Since $S$ is cute, it includes a copy of $s(v^*)$, so there is
a quasi--embedding $\psi: v^*\longrightarrow \omega$ of $s(v^*)$ into
$S$. Then, remembering $(\oplus)_4$,  we may choose two quasi--embeddings
$\pi_0,\pi_1: w\longrightarrow \rng(\psi)$ such that $\pi_0(\alpha)= \pi_1(
\alpha)$ for $\alpha\in\varphi_\xi[u]$ and $\pi_0\circ \varphi_\xi
=\psi\rest u^\xi$ and  $\pi_1\circ \varphi_\varsigma= \psi\rest
u^\varsigma$.  Apply Definiton \ref{defamal}(3) to $N,\varphi_\xi[u],
w,\pi_0,\pi_1$ to choose $K$ and a quasi--embedding
$\pi:\rng(\psi)\longrightarrow w^{\gb_K}$ such that
\[(\pi\circ\pi_0)(\gb_N\rest w)\Subset \gb_K\rest
  (\pi\circ\pi_0[w])\quad\mbox{ and }\quad
  (\pi\circ\pi_1)(\gb_N\rest w)\Subset \gb_K\rest (\pi\circ\pi_1[w]).\] 
Then the condition $p(n^{\gb_K},\pi\circ\psi,K)$ is a common upper bound of
$p_\xi,p_\varsigma$ (remember Claim \ref{cl3}(2)).  
\end{proof}

\begin{claim}
  \label{cl2}
The following sets are open dense in $\bbP$:

$D_\alpha=\{p\in\bbP:\alpha\in u^p\}$ for $\alpha<\lambda$, and 

$D^n=\{p\in\bbP:n^p>n\}$ for $n<\omega$.
\end{claim}

\begin{proof}[Proof of the Claim]
To show the density of $D_\alpha$ suppose that $p\in\bbP$ and $\alpha\in
\lambda\setminus u^p$. Let a quasi-embedding $\varphi:u^p\longrightarrow
\omega$ and $N<\omega$ witness $p\in\bbP$. Let $w=\rng(\varphi)$. 
Put $v^*=u^p\cup\{\alpha\}$ and let $s(v^*)$ be the finite
$\YZR(\vare)$--system associated with $v^*$. Since $S$ is cute, it includes a
copy of $s(v^*)$, so there is a quasi--embedding $\psi:v^*\longrightarrow
\omega$ of $s(v^*)$ into $S$. Applying \ref{defamal}(4) to $w$,
$\pi_0=\psi\circ \varphi^{-1}$ and $u=\rng(\psi)$ we may find $K<\omega$ and
a quasi--embedding $\pi:u\longrightarrow w^{\gb_K}$ such that 
\[(\pi\circ\pi_0)\big(\gb_N\rest w\big)\Subset \gb_K\rest (\pi\circ
  \pi_0[w]).\] 
Then the condition $p(n^{\gb_K},\pi\circ \psi,K)$ belongs to $D_\alpha$ and
it is stronger than $p$. 

To argue that $D^n$ is dense (for $n<\omega$) suppose that $p\in\bbP$ and
$n^p\leq n$. Let $\varphi:u^p\longrightarrow \omega$ and $N<\omega$ witness
that $p\in\bbP$. Let $M>n$ be such that $n^{\gb_M}>n$. Then the condition
$p(n^{\gb_M},\varphi,M)$ belongs to $D^n$ and it is stronger than $p$.
\end{proof} 

Now, for $(\alpha,\beta)\in\lambda^{\langle 2\rangle}$ we define
$\bbP$--names $\name{\eta}_\alpha$ and $\name{g}_i(\alpha, \beta)$ by 
\[\forces_\bbP\mbox{``} \name{\eta}_\alpha=\bigcup\{\eta^p_\alpha:
  \alpha\in u^p\ \wedge\ p\in\name{G}\}\mbox{ and }
\name{g}_i(\alpha,\beta)=\bigcup\{g^p_i(\alpha,\beta):
  \alpha,\beta\in u^p\ \wedge\ p\in\name{G}\}\mbox{''}.\]
By the definition of the order of $\bbP$ and by Claim \ref{cl2} we easily
see that 
\[  \begin{array}{ll}
\forces_\bbP &\mbox{`` } \langle \name{\eta}_\alpha:\alpha<\lambda\rangle 
  \subseteq \can \mbox{ are pairwise distinct,}\\
& \ \ \name{g}_i(\alpha,\beta)\in \bigcup\limits_{m<\omega} \lim(T_m) \mbox{ 
  for } (\alpha,\beta)\in \lambda^{\langle 2\rangle},\ i<\iota,\\
&\ \ \name{\eta}_\alpha+\name{\eta}_\beta=\name{g}_i(\alpha,\beta)+
  \name{g}_i(\beta,\alpha) \mbox{ for } (\alpha,\beta)\in \lambda^{\langle
  2\rangle},\ i<\iota \mbox{''}.
  \end{array}\] 
Hence $\bbP$ is as required.
\end{proof}

\section{Existence of $\Sigma^0_2$ sets with the amalgamation    
  property} 
Here we will prove our main result: there exists $\bar{T}$ with the
amalgamation property (over a cute $\YZR(\vare)$--system) and with the
nondisjointness rank $\ndrk$ bounded by $\omega\cdot (vare+2)+2$. For this
$\bar{T}$ (or rather $\bigcup\limits_{m<\omega} \lim(T_m)$) we may force
many $2\iota$--non-disjoint translations without adding a perfect set of
such translations. 

\begin{definition}
  \label{fmtkDef}
  Assume that $\bar{t}, M,\iota,n$ are such that 
  \begin{itemize}
\item $2\leq\iota<\omega$, and $M,n<\omega$, and 
\item $\bar{t}=\langle t_m:m<M\rangle$ where each $t_m\subseteq
  {}^{n\geq} 2$ is a tree with maximal nodes of length $n$ (for
  $m<M$).
 \end{itemize}
Let $\bar{T}^*=\langle T_m^*:m<\omega\rangle$ where for $m<M$ we 
have 
\[T_m^*=\Big\{\nu\in {}^{\omega>} 2:\nu\rest n\in t_m\ \wedge\
  \big(\forall k<\lh(\nu)\big)\big(n\leq k\ \Rightarrow\ \nu(k)=0\big) 
  \Big\}\] 
and   $T_m^* ={}^{\omega>} 2$ when $M\leq m<\omega$.
\begin{enumerate}
\item We say that $\bar{t}$ is {\em $(M,\iota,n)$--usable\/} if,  letting 
  $B=\bigcup\limits_{m<M} \lim(T_m^*)$ [ {\em sic} ], there are 
  pairwise different $\rho_0,\rho_1,\rho_2\in \can$ such that     
  \[\big|\big(\rho_j+B\big)\cap \big(\rho_{j'}+B\big)\big|\geq 2\iota\] 
for $j,j'<3$. 
\item We define $\fMtk$ as the set of all tuples
  $\bmm=(\ell_\bmm,u_\bmm, \bar{h}_\bmm, \bar{g}_\bmm)\in \Mtks$ such
  that $\ell_\bmm\leq n$ and $\rng(h^\bmm_i) 
  \subseteq M$ for each $i<\iota$. (Remember, $\Mtks$ was defined in
  Definition \ref{mtkDef} (for $\bar{T}^*,\iota$)). 
 \end{enumerate}
\end{definition}

\begin{observation}
If $\bmm\in \fMtk$ and $\rho\in {}^{\ell_\bmm}2$, then $\bmm+\rho \in\fMtk$
(see Definition \ref{traDef}).   
\end{observation}

\begin{lemma}
[See {\cite[Lemma 2.3]{RoRy18}}]
 \label{litlem}
Let $0<\ell<\omega$ and let $\cB\subseteq {}^\ell 2$ be a linearly 
independent set of vectors (in $({}^\ell2,+)$ over $(2,+_2,\cdot_2)$). If
$\cA\subseteq {}^\ell 2$, $|\cA|\geq 5$ and $\cA+\cA\subseteq \cB+\cB$,
then for a unique $x\in {}^\ell 2$ we have $\cA+x\subseteq \cB$.
\end{lemma}

\begin{theorem}
\label{rankset}
Assume $0<\vare<\omega_1$ and let $2\le \iota<\omega$. Let
$S=(\omega,\bar{r},\bar{j},\bar{k})$ be a cute $\YZR(\vare)$--system. Then   
there is a sequence $\bar{T}=\langle T_m:m<\omega\rangle$ of trees
$T_m\subseteq {}^{\omega>}2$ without maximal nodes and a sequence
$\bar{c}=\langle c_m:m<\omega\rangle$ of integers such that  
\begin{enumerate}
\item $\bar{T}$ has $(\bar{c},S)$--controlled amalgamation property, and 
\item $\vare\leq \ndrk(\bar{T})\leq\omega\cdot (\vare+2)+2$ (the
  ordinal multiplication).
\end{enumerate}
\end{theorem}

\begin{proof}
We will mix the forcing construction of \cite{RoSh:1138} with the
arguments of \cite{RoRy18}, getting our result for all $\iota\geq 2$. Let
$\cP_\iota$ be the collection of all tuples  
\[p=\big(w^p,n^p,M^p,\bar{\eta}^p,\bar{t}^p,\bar{d}^p,\bar{h}^p, \bar{g}^p,   
\cM^p,\bar{\rho}^p\big)= \big(w,n,M,\bar{\eta},\bar{t}, \bar{d},\bar{h},
\bar{g}, \cM, \bar{\rho} \big)\]      
such that the following demands $(\boxtimes)_1$--$(\boxtimes)_7$ are
satisfied.  
\begin{enumerate}
\item[$(\boxtimes)_1$] $w\in [\omega]^{<\omega}$, $|w|\geq 3$,
  $0<n,M<\omega$.  
\item[$(\boxtimes)_2$] $\bar{t}= \langle t_m:m<M\rangle$ is
  $(M,\iota,n)$--usable, so in particular $\emptyset\neq t_m\subseteq
  {}^{n\geq} 2$ (for $m<M$) is a tree in which all terminal branches are of
  length $n$.  
\item[$(\boxtimes)_3$] $\bar{d}=\langle d_m:m<M\rangle$, where $0<d_m\leq n$
for $m<M$. 
\item[$(\boxtimes)_4$] $\gb(p)=(w^p,n^p,\bar{\eta}^p, \bar{h}^p, \bar{g}^p,
  \cM^p)$ is an $(S,\iota,\bar{t},\bar{d})$--brick (cf Definition
  \ref{brick}(2) and Remark \ref{rembri}(3)).
\item[$(\boxtimes)_5$] $\bar{\rho}=\langle \rho_{i,a,b}: i<\iota,\ a,b\in
  w,\ a<b\rangle\subseteq {}^n2$ and 
\[g_i(a ,b )=\eta_a+\rho_{i,a,b}\mbox{ and  } g_i(b,a)=\eta_b+\rho_{i,a,b}\]
whenever $a<b$ are from $w$ and $i<\iota$.
\item[$(\boxtimes)_6$] the list
\[\bar{\eta}\conc\bar{\rho}=\langle \eta_a:a\in w\rangle\conc \langle
  \rho_{i,a ,b}: i<\iota,\ a ,b\in  w,\ a<b\rangle\]
is a list of linearly independent vectors (in $({}^n 2,+,\cdot)$  over 
$(2,+_2,\cdot_2)$); in particular they are pairwise distinct,     
\item[$(\boxtimes)_7$] if $m<M$ then 
\[t_m\cap {}^n 2\subseteq \{g_i(a ,b ):(a ,b )\in w^{\langle 
    2\rangle}\ \wedge\ i<\iota\},\] 
and $t_m\cap t_{m'}\cap {}^n 2=\emptyset$ whenever $m<m'<M$. 
\end{enumerate}
For $p,q\in \cP_\iota$ we declare that $p\curlyeqprec q$\quad if and only if 
\begin{itemize}
\item $n^p\leq n^q$, $M^p\leq M^q$, and $t^p_m=t^q_m\cap {}^{n^p\geq} 2$ and
  $d^p_m=d^q_m$ for all $m<M^p$, and     
\item $\gb(p)\Subset \gb(q)$. 
\end{itemize}
It is straightforward to verify that $(\cP_\iota,\curlyeqprec)$ is a
nonempty partial order.   
\medskip

\begin{claim}
\label{cl8}
Assume $p=\big(w,n,M,\bar{\eta},\bar{t}, \bar{d},\bar{h},
\bar{g}, \cM, \bar{\rho} \big)\in \cP_\iota$. Suppose that 
$\nu^0_i,\nu^1_i\in\bigcup\limits_{m<M} (t_m\cap  {}^n2)$ (for
$i<\iota$) are such that    
  \begin{enumerate}
  \item[(a)] there are no repetitions in $\langle \nu^0_i,\nu^1_i: 
    i<\iota\rangle$, and 
  \item[(b)] $\nu^0_i+\nu^1_i=\nu^0_j+\nu^1_j$ for $i<j<\iota$.
  \end{enumerate}
  Then
  \begin{enumerate}
  \item[(A)] if $\iota\geq 3$ then for some $ a , b \in w$ we have  
\[\big\{\{\nu^0_i,\nu^1_i\}:i<\iota\big\}=\big\{\{g_i(a,b),  g_i(b,a)\}:
  i<\iota\big\}.\]
\item[(B)] If $\iota=2$ then for some $a,b\in w$ we have   
\[\big\{\nu^0_0,\nu^1_0,\nu^0_1,\nu^1_1\big\}=\big\{g_0(a,b),  g_0(b,a), 
  g_1(a,b),  g_1(b,a)\big\}.\] 
  \end{enumerate}
\end{claim}

\begin{proof}[Proof of the Claim]
For $a>b$ from $w$ and $i<\iota$ we will write $\rho_{i,a,b}$ for
$\rho_{i,b,a}$. With this notation, all elements of
$\bigcup\limits_{m<M}(t_m\cap  {}^n2)$ are of the form
$\eta_a+\rho_{i,a,b}$ for some $i<\iota$ and $(a,b)\in w^{\langle
  2\rangle}$. 

Let $i<j<\iota$ and let $a,b,c,d, a',b',c',d', i_0,i_1, i_0',i_1'$ be such
that $\nu^0_i=\eta_a+\rho_{i_0,a,b}$, $\nu^1_i=\eta_c+ \rho_{i_1,c,d}$, 
$\nu^0_j=\eta_{a'}+\rho_{i_0',a',b'}$, $\nu^1_j=\eta_{c'}+
\rho_{i_1',c',d'}$. Since $\nu^0_i+\nu^1_i=\nu^0_j
+\nu^1_j$ we have then
\[\eta_a+\rho_{i_0,a,b}+\eta_c+\rho_{i_1,c,d}=
  \eta_{a'}+\rho_{i_0',a',b'}+\eta_{c'}+\rho_{i_1',c',d'}.\]
If $a\neq c$ then it follows from $(\boxtimes)_6$ that $\{a,c\}=\{a',c'\}$
(so either $a=a'$, $c=c'$ or $a=c'$, $c=a'$), and    
\[\rho_{i_0,a,b}+\rho_{i_1,c,d}=\rho_{i_0',a',b'}+\rho_{i_1',c',d'}.\]
Then, still assuming $a\neq c$, we consider relationships among $\rho$'s 
above getting four possible subcases.\\
If $\rho_{i_0,a,b}=\rho_{i_1,c,d}$ then $a\in \{a,b\}=\{c,d\}\ni c$ and also
$\rho_{i_0',a',b'}= \rho_{i_1',c',d'}$ so $a'\in\{a',b'\} =\{c',d'\}\ni
c'$. Moreover, $i_0=i_1$ and $i_0'=i_1'$. Thus, remembering that
$\{a,c\}=\{a',c'\}$,  we get in this case: 
\begin{enumerate}
\item[$\big({\boldsymbol \rightrightarrows}\big)_{i,j}^{a,c}$] if $a=a'$ and 
  $c=c'$,  then also $c=b=b'$, $a=d=d'$ and 
\[\begin{array}{ll}
\nu^0_i=\eta_a+\rho_{i_0,a,c},&\ \nu^1_i=\eta_c+\rho_{i_0,a,c},\\ 
  \nu^0_j=\eta_a+\rho_{i_0',a,c}, &\ \nu^1_j=\eta_c+\rho_{i_0',a,c},
\end{array}\] 
\item[\ \ ]  if $a=c'$
  and $c=a'$, then also $a=d=b'$ and $c=d'=b$ and  
\[\begin{array}{ll}
\nu^0_i=\eta_a+\rho_{i_0,a,c},&\ \nu^1_i=\eta_c+\rho_{i_0,a,c},\\
  \nu^0_j=\eta_c+\rho_{i_0',a,c},& \ \nu^1_j=\eta_a+\rho_{i_0',a,c}.
\end{array}\] 
\end{enumerate}
If $\rho_{i_0,a,b}\neq \rho_{i_1,c,d}$ then  $\{\rho_{i_0,a,b},
\rho_{i_1,c,d}\}= \{\rho_{i_0',a',b'}, \rho_{i_1',c',d'}\}$ and analysis as
above provides that there are only two possible cases. 
\begin{enumerate}
\item[$\big({\boldsymbol \downdownarrows}\big)_{i,j}^{a,c}$] If
  $\rho_{i_0,a,b}=\rho_{i_0',a',b'}$ then we must also have $a=c'$ and
\[\begin{array}{ll}
\nu^0_i=\eta_a+\rho_{i_0,a,c},&\ \nu^1_i=\eta_c+\rho_{i_1,a,c},\\ 
  \nu^0_j=\eta_c+\rho_{i_0,a,c}, &\ \nu^1_j=\eta_a+\rho_{i_1,a,c}.
\end{array}\] 
\item[$\big({\boldsymbol \searrow}\hspace{-12pt}
{\boldsymbol
    \swarrow}\big)_{i,j}^{a,c}$] If $\rho_{i_0,a,b} = \rho_{i_1',c',d'}$
  then we must also have $a=a'$ and  
\[\begin{array}{ll}
\nu^0_i=\eta_a+\rho_{i_0,a,c},&\ \nu^1_i=\eta_c+\rho_{i_1,a,c},\\
  \nu^0_j=\eta_a+\rho_{i_1,a,c},& \ \nu^1_j=\eta_c+\rho_{i_0,a,c}.
\end{array}\] 
\end{enumerate}
Now about what happens if $a=c$ (and $a'=c'$). We easily eliminate
the possibility of $\rho_{i_0,a,b}=\rho_{i_1,c,d}$. Considering all other
options we get the following.
\begin{enumerate}
\item[$\big({\boldsymbol \downdownarrows}\big)_{i,j}^{a,a'}$] If
  $\rho_{i_0,a,b}=\rho_{i_0',a',b'}$ then 
\[\begin{array}{ll}
\nu^0_i=\eta_a+\rho_{i_0,a,a'},&\ \nu^1_i=\eta_a+\rho_{i_1,a,a'},\\ 
  \nu^0_j=\eta_{a'}+\rho_{i_0,a,a'}, &\ \nu^1_j=\eta_{a'}+\rho_{i_1,a,a'}.
\end{array}\] 
\item[$\big({\boldsymbol \searrow}\hspace{-12pt}
{\boldsymbol
    \swarrow}\big)_{i,j}^{a,a'}$] If $\rho_{i_0,a,b} = \rho_{i_1',c',d'}$ then 
\[\begin{array}{ll}
\nu^0_i=\eta_a+\rho_{i_0,a,a'},&\ \nu^1_i=\eta_a+\rho_{i_1,a,a'},\\
  \nu^0_j=\eta_{a'}+\rho_{i_1,a,a'},& \ \nu^1_j=\eta_{a'}+\rho_{i_0,a,a'}.
\end{array}\] 
\end{enumerate}
Thus we see that, for each $i<j<\iota$ we have
\begin{enumerate}
\item[$(\heartsuit)_{i,j}$] there are $a<b$ from $w$ and $i_0,i_1<\iota$ such
  that
  \[\big\{\nu^0_i,\nu^1_i,\nu^0_j,\nu^1_j\big\}= \big\{g_{i_0}(a,b),
    g_{i_0}(b,a), g_{i_1}(a,b), g_{i_1}(b,a)\big\}.\] 
\end{enumerate}
This immediately gives us the assertion of (B). If $\iota\geq 3$ then
considering triples $i<j<k<\iota$ and $(\heartsuit)_{i,j} +(\heartsuit)_{i,k}
+(\heartsuit)_{k,j}$ we get  from the linear independence declared in
$(\boxtimes)_6$ that 
\begin{enumerate}
\item[$(\heartsuit)^+$] for some $a<b$ from $w$, for every $i<\iota$ we have 
\[\nu^0_i,\nu^1_i\in \big\{g_j(a,b), g_j(b,a): j<\iota\big\}=\big\{
  \eta_a+\rho_{j,a,b},\eta_b+\rho_{j,a,b}:j<\iota \big\}.\]
\end{enumerate}
By the same linear independence, 
\begin{itemize}
\item the sum $g_{i_0}(a,b)+g_{j_0}(b,a)$ (where $i_0\neq j_0$) can be equal
  to only one  other sum of two elements of $\big\{g_j(a,b), g_j(b,a):
  j<\iota\big\}$, namely $g_{i_0}(b,a)+g_{j_0}(a,b)$,
\item the sum $g_{i_0}(a,b)+g_{j_0}(a,b)$ (for $i_0\neq j_0$) can be equal
  to only one  other sum of two elements of $\big\{g_j(a,b), g_j(b,a):
  j<\iota\big\}$, namely $g_{i_0}(b,a)+g_{j_0}(b,a)$.
\end{itemize}
Therefore, if $\iota\geq 3$ then for $a,b$ given by $(\heartsuit)^+$,
\[\big\{\{\nu^0_i,\nu^1_i\}:i<\iota\big\} =\big\{\{g_j(a,b), g_j(b,a)\}: j<
  \iota \big\}\]
and the assertion of (A) follows.
\end{proof} 

\begin{claim}
\label{cl7}
Let $p=\big(w^p,n^p,M^p,\bar{\eta}^p,\bar{t}^p, \bar{d}^p, \bar{h}^p,
\bar{g}^p, \cM^p, \bar{\rho}^p \big)\in\cP_\iota$. Assume that $\bmm\in
\fMtp$ is such that   
\begin{enumerate}
\item[(i)] $|u_\bmm|\geq 5$, and 
\item[(ii)] $d_{h^\bmm_i(\eta,\nu)}\leq \ell_\bmm$ for all $(\eta,\nu)\in 
  \big(u_\bmm\big)^{\langle 2\rangle}$ and $i<\iota$.
\end{enumerate}
Then for some $\rho\in {}^{\ell_\bmm} 2$ and $\bn\in\cM^p$ we have 
$(\bmm+\rho)\doteqdot \bn$.    
\end{claim}

\begin{proof}[Proof of the Claim]
Suppose $\bmm\in\fMtp$ satisfies (i) and (ii). By Definitions
\ref{fmtkDef}(2) +  \ref{mtkDef}(f) there is $\bmm^+\in\fMtp$ such that 
$\bmm\sqsubseteq \bmm^+$, $\ell_{\bmm^+}=n^p$ and
$|u_{\bmm^+}|=|u_\bmm|$. If $(\eta,\nu)\in (u_{\bmm^+})^{\langle 2\rangle}$
then  for all $i<\iota$:  
\[g_i^{\bmm^+}(\nu,\eta), g_i^{\bmm^+}(\eta,\nu) \in \bigcup_{m<M^p}
  t^p_m\cap {}^{n^p} 2\quad \mbox{ and }\quad g_i^{\bmm^+}(\nu,\eta)+ 
  g_i^{\bmm^+}(\eta,\nu) = \nu+\eta.\]
Let us consider {\bf the case when $\iota\geq 3$}. Then by Claim
\ref{cl8}(A), for every $(\eta,\nu)\in (u_{\bmm^+})^{\langle 2\rangle}$
there are $a<b$ from $w^p$ such that   
\begin{enumerate}
\item[$(*)^{\nu,\eta}_{ a , b }$] \qquad $\big\{\{g_i^{\bmm^+}(\nu,\eta),
  g_i^{\bmm^+}(\eta,\nu)\} : i<\iota\big\}=\big\{\{g_i^p( a , b ),
  g_i^p( b , a )\} : i<\iota\big\}$.
\end{enumerate}
In particular, $\eta+\nu=\eta_ a^p +\eta_ b^p $. Using Lemma
\ref{litlem} we may conclude that for some $x\in {}^{n^p} 2$ we have
$u_{\bmm^+}+x\subseteq \{\eta_c^p:c\in w^p\}$. The linear
independence of $\eta_c^p$'s implies that if $\eta,\nu\in
(u_{\bmm^+})^{\langle 2\rangle}$ and $(*)^{\nu,\eta}_{ a , b }$ holds,
then $\{\eta+x,\nu+x\}=\{\eta_ a^p ,\eta_ b^p \}$. By $(\boxplus)_7$,
$g^p_i( a , b )$ ($g^{\bmm^+}_i(\nu,\eta)$, respectively) determines
$h^p_i( a , b )$  ($h^{\bmm^+}_i(\nu,\eta)$, respectively). So easily   
$(\bmm^++x)\doteqdot \bn^+$ for some $\bn^+\in\cM^p$ and
\[\bmm=\bmm^+\rest \ell_\bmm\doteqdot (\bn^+\rest \ell_\bmm) + x\rest  
  \ell_\bmm\] 
and $\bn^+\rest \ell_\bmm\in\cM^p$ (remember assumption (ii) for $\bmm$).
\medskip

Now, consider {\bf the case when $\iota=2$}. By Claim \ref{cl8}(B), we know
that for every $(\eta,\nu)\in (u_{\bmm^+})^{\langle 2\rangle}$ 
\begin{enumerate}
\item[$(**)^{\nu,\eta}$]  \qquad there are $ a < b $ from $w^p$ such that  
\[\big\{g_0^{\bmm^+}\!(\nu,\eta),  g_0^{\bmm^+}\!(\eta,\nu),
  g_1^{\bmm^+}\!(\nu,\eta), g_1^{\bmm^+}\!(\eta,\nu) \big\}= \big\{g_0^p(a,b),
  g_0^p(b,a), g_1^p(a,b), g_1^p(b,a)\big\}.\] 
\end{enumerate}
Define functions $\chi:[u_{\bmm^+}]^{2}\longrightarrow 2$ and
$\Theta:[u_{\bmm^+}]^{2}\longrightarrow [w^p]^2$ as follows. Suppose
$\{\eta,\nu\}\in u_{\bmm^+}$. 
\begin{itemize}
\item If $\eta+\nu=\eta_a^p+\eta_b^p$, $a,b\in w^p$, then $\chi(\{\eta,\nu\})=1$
  and $\Theta(\{\eta,\nu\})=\{a,b\}$.
\item If $\eta+\nu=\eta_a^p+\eta_b^p+\rho^p_{0,a,b}+\rho^p_{1,a,b}$, $a,b\in
  w^p$, then $\chi(\{\eta,\nu\})=0$ and $\Theta(\{\eta,\nu\})=\{a,b\}$.
\item If $\eta+\nu=\rho_{0,a,b}^p+\rho_{1,a,b}^p$, $a,b\in w^p$, then
  $\chi(\{\eta,\nu\})=0$ and $\Theta(\{\eta,\nu\})=\{a,b\}$.  
\end{itemize}
It follows from $(**)^{\nu,\eta}$ and the linear independence of
$\bar{\eta}^p\conc \bar{\rho}^p$ (see $(\boxtimes)_6$) that exactly one of the
cases described above holds for $\eta+\nu$.

\begin{enumerate}
\item[$(***)_1$] If $\eta_0,\eta_1,\eta_2\in u_{\bmm^+}$ are pairwise distinct
  and $\chi(\{\eta_0,\eta_1\})=\chi(\{\eta_1,\eta_2\})=1$, then
  $\Theta(\{\eta_0,\eta_1\})\neq \Theta(\{\eta_1,\eta_2\})$ and
  $\chi(\{\eta_0,\eta_2\})=1$.
\end{enumerate}
Why? Assume $\chi(\{\eta_0,\eta_1\})=\chi(\{\eta_1,\eta_2\})=1$. Then both
$\eta_0+\eta_1$ and $\eta_1+\eta_2$ are sums of two elements of
$\{\eta^p_c:c\in w^p\}$. Hence $\eta_0+\eta_2$ is a sum of some elements of 
$\{\eta^p_c:c\in w^p\}$ and therefore $\chi(\{\eta_0,\eta_2\})\neq 0$ (as the
terms of $\bar{\eta}^p\conc\bar{\rho}^p$ are linearly independent). Now, if we
had $\Theta(\{\eta_0,\eta_1\})=\Theta(\{\eta_1,\eta_2\})=\{a,b\}$, then
\[\eta_0+\eta_1=\eta^p_a+\eta^p_b=\eta_1+\eta_2\]
and hence $\eta_0=\eta_2$, a contradiction. 
\begin{enumerate}
\item[$(***)_2$] If $\eta_0,\eta_1,\eta_2\in u_{\bmm^+}$ are pairwise distinct
  and $\chi(\{\eta_0,\eta_1\})=\chi(\{\eta_0,\eta_2\})=0$, then
  $\Theta(\{\eta_0,\eta_1\})=\Theta(\{\eta_0,\eta_2\})=\Theta(\{\eta_1,\eta_2\})$
  and  $\chi(\{\eta_1,\eta_2\})=1$.
\end{enumerate}
Why? First note that if we had $\Theta(\{\eta_0,\eta_1\})\neq
\Theta(\{\eta_0,\eta_2\})$ then $\eta_1+\eta_2=(\eta_0+\eta_1)+
(\eta_0+\eta_2)$ would be a sum of four elements of $\{\rho^p_{i,a,b}: i<2,\
a<b\mbox{ from }w\}$ and possiby some elements of $\{\eta^p_c: c\in
w^p\}$. This is clearly impossible and thus $\Theta(\{\eta_0,\eta_1\})=
\Theta(\{\eta_0,\eta_2\})$, say it is $\{a,b\}$. Since $\eta_0+\eta_1\neq
\eta_0+\eta_2$ we immediately conlude that one of them is
$\eta^p_a+\eta^p_b+\rho^p_{0,a,b} +\rho^p_{1,a,b}$ and the other is  
$\rho^p_{0,a,b} +\rho^p_{1,a,b}$. Consequently,
\[\eta_1+\eta_2=(\eta_0+\eta_1)+(\eta_0+\eta_2)=\eta^p_a+\eta^p_b.\] 
\begin{enumerate}
\item[$(***)_3$] If $\eta_0,\eta_1,\eta_2,\eta_3\in u_{\bmm^+}$ are pairwise
  distinct and $\chi(\{\eta_0,\eta_1\})=\chi(\{\eta_0,\eta_2\})=0$, then
$\chi(\{\eta_0,\eta_3\})=1$.
\end{enumerate}
Why? Assume towards contradiction that $\chi(\{\eta_0,\eta_3\})=0$. It follows
from $(***)_2$ that 
\[\Theta(\{\eta_1,\eta_2\})=\Theta(\{\eta_0,\eta_1\})=\Theta(\{\eta_0,
  \eta_2\})=   \Theta(\{\eta_0,\eta_3\})= \Theta(\{\eta_2,\eta_3\})\]
and $\chi(\{\eta_1,\eta_2\})=\chi(\{\eta_2,\eta_3\})=1$. Thus, letting
$\{a,b\} =\Theta(\{\eta_2,\eta_3\})$, we have $\eta_2+\eta_3=\eta^p_a+\eta^p_b
= \eta_1+\eta_2$, a contradiction.
\begin{enumerate}
\item[$(***)_4$] $\chi(\{\eta_0,\eta_1\})=1$  for all distinct
  $\eta_0,\eta_1\in u_{\bmm^+}$.
\end{enumerate}
Why? Suppose towards contradiction that $\chi(\{\eta_0,\eta_1\})=0$. It
follows from $(***)_3$ that there is at most one $\eta\in
u_{\bmm^+}\setminus \{\eta_0,\eta_1\}$ such that $\chi(\{\eta_0,\eta\})=0$, 
and there is at most one $\eta\in u_{\bmm^+}\setminus \{\eta_0,\eta_1\}$ such
that $\chi(\{\eta_1,\eta\})=0$. Since $|u_{\bmm^+}|\geq 5$ we may choose
$\eta_2\in u_{\bmm^+}\setminus \{\eta_0,\eta_1\}$ such that
$\chi(\{\eta_0,\eta_2\})=\chi(\{\eta_1,\eta_2\})=1$. Then, however, we get an
immediate contradiction with $(***)_1$.

Consequently,
\[\big(\forall\eta,\nu\in u_{\bmm^+}\big)\big(\exists a,b\in w^p\big)
  \big(\eta+\nu=\eta^p_a+\eta^p_b\big),\]
and we may get our desired conclusion similarly to the case of $\iota\geq
3$. (Notice the difference in Definition \ref{almostDef} of $\doteqdot$ for
case $\iota=2$.) 
\end{proof}

\begin{claim}
  \label{cl5}
  Assume that 
\begin{enumerate}
\item[(a)] $p\in\cP_\iota$ and $u\subseteq w\subseteq w^p$, $|w|\geq 3$, and $w^*\in
  \big[\omega\setminus w^p\big]^{<\omega}$,
\item[(b)] $\bar{k}\big(v\cup\{d\}\big)\neq |d\cap 
  v|$ whenever $v\subseteq u$ and $d\in w\setminus u$ and 
  $\bar{r}\big(v\cup\{d\}\big)=0$, 
\item[(c)]   $\pi_0,\pi_1:w \longrightarrow w^*$ are quasi--embeddings 
  (into $S$) such that $\pi_0( a )=\pi_1( a )$  for $ a \in u$ and
  $\pi_0[w\setminus u]\cap \pi_1[w\setminus u]=\emptyset$. 
\end{enumerate}
Then there is $q\in\cP_\iota$ such that $p \curlyeqprec q$, $w^q=w^*\cup w^p$ and  
\[\pi_0(\gb(p)\rest w)\Subset \gb(q)\rest
  (\pi_0[w])\quad\mbox{ and }\quad
  \pi_1(\gb(p)\rest w)\Subset \gb(q)\rest (\pi_1[w]).\] 
\end{claim}

\begin{proof}[Proof of the Claim]
  Let $N=|w^p|+|w^*|$, $K=\big(|w^p|+|w^*|\big)^2$, and
  \[K^*= \Big|\big(w^p\cup w^*\big)^{\langle 2\rangle}\setminus \Big(\big(
    w^p\big)^{\langle 2\rangle}\cup\big(\pi_0[w] \big)^{\langle 2\rangle}\cup
    \big(\pi_1[w]\big)^{\langle 2\rangle}\Big)\Big|.\]
Fix {\bf injections }
\[\psi_0:w^p\cup w^*\longrightarrow [n^p,n^p+N),\quad
\psi_1: \iota\times\big(w^p\cup w^*\big)^{\langle 2\rangle} \longrightarrow 
[n^p+N,n^p+N+\iota\cdot K)\]
and
\[\varphi: \big(w^p\cup w^*\big)^{\langle 2\rangle}\setminus \Big(\big(
    w^p\big)^{\langle 2\rangle}\cup\big(\pi_0[w] \big)^{\langle 2\rangle}\cup
    \big(\pi_1[w]\big)^{\langle 2\rangle}\Big) \longrightarrow
    [M^p,M^p+K^*).\] 
Define:\\
$w^q=w^p\cup w^*$, $n^q=n^p+N+\iota\cdot K$, $M^q=M^p+K^*$,\\
$\bar{\eta}^q=\langle \eta^q_ a : a \in w^q\rangle$ and
\begin{itemize}
\item  if $ a \in w^p$ then $\eta^q_a\rest n^p=\eta^p_a$,
$\eta^q_a(\psi_0(a))=1$ and $\eta^q_a(\ell)=0$ for all other
  $\ell\in [n^p,n^q)$,
\item if $j<2$, $ a =\pi_j(c)\in \pi_j[w]$,  $c\in w$,  then
$\eta^q_a\rest n^p=\eta^p_c$,  $\eta^q_a(\psi_0(a))=1$
and $\eta^q_a(\ell)=0$ for all other $\ell\in [n^p,n^q)$ (note that
  by assumption (c), if $ a =\pi_j(c)$, $c\in u$, then also 
  $ a =\pi_{1-j}(c)$, so there is no ambiguity here),
\item if $ a \in w^*\setminus (\pi_0[w]\cup \pi_1[w])$, then
$\eta^q_a(\psi_0(a))=1$ and $\eta^q_a(\ell)=0$ for all other
$\ell<n^q$,
\end{itemize}
$\bar{h}^q=\langle h^q_i:i<\iota\rangle$ and for $i<\iota$ and
$(a,b)\in (w^q)^{\langle 2\rangle}$:
\begin{itemize}
\item if $(a,b)\in (w^p)^{\langle 2\rangle}$, then
   $h^q_i(a,b)= h^p_i(a,b)$,
\item if $j<2$, $(a,b)\in \big(\pi_j[w]\big)^{\langle 2\rangle}$,
  and $ a =\pi_j(c)$, $ b =\pi_j(d)$ where $c,d\in  
 w$, then  $h^q_i(a,b)= h^p_i(c,d)$,   
\item if $(a,b)\in \big( w^q\big)^{\langle 2\rangle}\setminus
\Big(\big( w^p\big)^{\langle 2\rangle}\cup\big(\pi_0[w] \big)^{\langle
2\rangle}\cup \big(\pi_1[w]\big) ^{\langle 2\rangle}\Big)$, then 
$h^q_i(a,b)=\varphi(a,b)$,
\end{itemize}
$\bar{\rho}^q=\langle \rho^q_{i,a,b}:i<\iota,\ a,b\in w^q,\ a<b\rangle$ and
for $i<\iota$ and $a<b$ from $w^q$: 
\begin{itemize}
\item if $a,b\in w^p$ then $\rho^q_{i,a,b}\in {}^{n^q} 2$ is such that
  $\rho^p_{i,a,b}\trianglelefteq \rho^q_{i,a,b}$, $\rho^q_{i,a,b}
  (\psi_1(i,a,b))=1$ and $\rho^q_{i,a,b}(\ell)=0$ for all other $\ell<n^q$,
\item  if $j<2$, $(a,b)\in \big(\pi_j[w]\big)^{\langle 2\rangle}$,
  and  $a =\pi_j(c)$, $ b =\pi_j(d)$ where $c,d\in  w$, then
  $\rho^q_{i,a,b}\in {}^{n^q}2$ is such that $\rho^p_{i,c,d} \trianglelefteq
  \rho^q_{i,a,b}$, $\rho^q_{i,a,b} (\psi_1(i,a,b))=1$ and $\rho^q_{i,a,b}(\ell)=0$ for
  all other $\ell<n^q$, 
\item if   $(a,b)\in \big( w^q\big)^{\langle 2\rangle}$ is not
  covered by the cases above, then $\rho^q_{i,a,b} (\psi_1(i,a,b))=1$ and
  $\rho^q_{i,a,b}(\ell)=0$ for all other $\ell<n^q$, 
\end{itemize}
$\bar{g}^q$ is defined by condition $(\boxtimes)_5$ and $\cM^q$
is defined by Definition \ref{brick}$(\boxplus)_5$,
$\bar{t}^q=\langle t^q_m:m<M^q\rangle$ is such that
\[t^q_m=\{g^q_i(a,b)\rest n: n\leq n^q,\ i<\iota,\ (a,b)\in
  \big(w^q\big)^{\langle 2\rangle},\ h^q_i(a,b)=m\},\]
$\bar{d}^q=\langle d^q_m:m<M^q\rangle$, where $d^q_m=d^p_m$ if
$m<M^p$ and $d^q_m=n^q$ if $M^p\leq m<M^q$.
\bigskip

The verification that $q=\big(w^q,n^q,M^q,\bar{\eta}^q,\bar{t}^q, \bar{d}^q,
\bar{h}^q, \bar{g}^q, \cM^q\big)\in\cP_\iota$ is quite straightforward. The only
non trivial part is checking conditions $(\boxplus)_6$ and $(\boxplus)_7$ of
Definition \ref{brick}.  For $(\boxplus)_6$, assume that $\bmm^{\gb(q)}(\ell,w_0),
\bmm^{\gb(q)}(\ell,w_1)\in \cM^q$ are such that $\bmm^{\gb(q)}(\ell,w_0)
\doteqdot\bmm^{\gb(q)}(\ell,w_1)+\rho$. If for some $( a , b )\in
(w_0)^{\langle 2\rangle}$ and $i<\iota$ we have $h^q_i( a , b )\geq
M^p$, then $n^q\geq \ell\geq d^q_{h^*_i(a,b)}=n^q$ and $\{\eta^q_a: a\in
w_0\}= \{\eta^q_a:a\in w_1\}$ (and $|w_0|\geq 3$). Hence by the linear
independence of $\bar{\eta}^q$ we may conlude that $\rho=0$ and
$w_0=w_1$. Suppose now that $h^q_i( a , b )<M^p$ for all $( a , b )\in
(w_0)^{\langle   2\rangle}$. Then, both for $j=0$ and $j=1$, neccessarily
either we have $w_j\subseteq w^p$ or $w_j\subseteq \pi_0[w]$ or
$w_j\subseteq \pi_1[w]$. If $w_0\cup w_1\subseteq w^p$ then we may set
$\ell^*=\min(n^p,\ell)$ and apply condition $(\boxplus)_6$ for $p$ and
$\bmm^{\gb(p)}(\ell^*,w_0)$, $\bmm^{\gb(p)}(\ell^*,w_1)$. If $w_0\subseteq
w^p$ and $w_1\subseteq \pi_j[w]$,  then we first note that for each $n\in
[n^p,n^q)$ there is at most one $a\in w^q$ such that
$\eta^q_a(n)=1$. Therefore, in the current situation, $\eta^q_a(n)=0$
whenever $n^p\leq n\leq \ell$, $a\in w_0\cup w_1$. So we set
$\ell^*=\min(n^p,\ell)$ and again apply condition $(\boxplus)_6$ for $p$ and
$\bmm^{\gb(p)}(\ell^*,w_0)$, $\bmm^{\gb(p)}(\ell^*,
\pi_j^{-1}\big[w_1\big])$. Similarly in other cases.       

To show $(\boxplus)_7$ of Definition \ref{brick}, suppose towards
contradiction that $\bmm^{\gb(q)}(\ell_0,w_0)$ and $
\bmm^{\gb(q)}(\ell_1,w_1)$ from $\cM^q$ and $ a \in w_0$ are such that
$\bmm^{\gb(q)}(\ell_0,w_0) \sqsubseteq^* \bmm^{\gb(q)}(\ell_1,w_1)$,
$\bar{r}(w_0)=0$, $| a \cap  w_0|=\bar{k}(w_0)$ and $1<|\{ b \in
w_1:(\eta^q_ a \rest \ell_0)\vtl \eta^q_ b \}|$. We may assume that $4\leq
|w_0|+1=|w_1|$ and $\{ b _0, b _1\}=\{ b \in w_1:(\eta^q_ a \rest
\ell_0)\vtl \eta^q_ b \}$. Necessarily, $\ell_0<\ell_1\leq n^q$ and
therefore $h^q_i( a , b )<M^p$ for all $( a , b )\in (w_0)^{\langle 
  2\rangle}$ and $i<\iota$. Consequently, either $w_0\subseteq w^p$ or
$w_0\subseteq \pi_0[w]$ or $w_0\subseteq \pi_1[w]$. If also $w_1\subseteq
w^p$ or $w_1\subseteq \pi_0[w]$ or $w_1\subseteq \pi_1[w]$, then for any
distinct $a,b\in w_1$ we have $\eta_a^q\rest n^p\neq \eta^q_b\rest n^p$ and
we may assume $\ell_0<\ell_1\leq n^p$. Then we may use $\pi_0^{-1},
\pi_1^{-1}$, as appropriate, to copy (if needed) both $w_0$ and $w_1$ to
$w^p$ and get easy contradiction with $(\boxplus)_7$ for $p$. 

So suppose  otherwise, that is neither of the inclusions 
\[w_1\subseteq w^p,\quad  w_1\subseteq \pi_0[w],\quad w_1\subseteq
  \pi_1[w]\] 
holds true. We may replace $w_0$ with $\pi_j^{-1}[w_0]$ and $\ell_0$ with
$\min(n^p,\ell_0)$, so without loss of generality $w_0\subseteq w^p$ and
$\ell_0\leq n^p$. Now, for all  $ b \in w_1\setminus\{ b _0, b _1\} 
\neq\emptyset$ we have 
\[h^q_0( b , b _0), h^q_0( b , b _1)<M^p\quad \mbox{ while }\quad
  h^q_0( b _0, b _1)\geq M^p.\]
This is only possible if $ b _0\in \pi_j[w]\setminus \pi_{1-j}[w]$ and
$ b _1\in \pi_{1-j}[w]\setminus \pi_j[w]$ (say $ b _j\in
\pi_j[w]\setminus \pi_{1-j}[w]$) and $w_1\setminus
\{ b _0, b _1\}\subseteq \pi_0[w]\cap \pi_1[w]=\pi_0[u]=\pi_1[u]$. But
then letting $v=\pi_0^{-1}\big[ w_1\setminus\{ b _0, b _1\}\big]$ and
considering $\bmm^{\gb(p)}(\ell_0, v\cup\{\pi^{-1}_0( b _0)\})$ we get (by
\ref{brick}$(\boxplus)_6$ for $p$) that the order isomorphism from $w_0$ onto
$v\cup\{\pi^{-1}_0( b _0)\}$ is a quasi--embedding mapping $ a $ to
$\pi_0^{-1}( b _0)\in w\setminus u$. This immediately contradicts
assumption (b) of the Claim (applied to $v$ and $d=\pi_0^{-1}(b _0)$).

Thus $q\in\cP_\iota$ indeed. It should be clear that $p\curlyeqprec q$ and $q$ is
as required.
\end{proof}
\medskip

\begin{claim}
  \label{cl14}
  Let $p\in\cP_\iota$, $k<\omega$. Then there is $q\in \cP_\iota$ such that
  $p\curlyeqprec q$, $k\in w^q$, $n^q>k$ and $M^q>k$.
\end{claim}

\begin{proof}[Proof of the Claim]
  By the cuteness of $S$, we may find a quasi-embedding
  $\varphi:w^p\longrightarrow \omega$ such that $\max(w^p)<
  \min(\rng(\varphi))$. Let $w^*\in [\omega\setminus w^p]^{<\omega}$ be such
  that $k\in w^p\cup w^*$, $\rng(\varphi)\subseteq w^*$ and
  $|w^*|>k$. Applying (the proof of) Claim \ref{cl5} to $p$, $u=w=w^p$,
  $w^*$ and $\pi_0=\pi_1=\varphi$ we will get $q\in\cP_\iota$ as there.
(Note that the assumption \ref{cl5}(b) is satisfied vacuously.) This $q$ has
the properties that $p\curlyeqprec q$, and $w^q=w^p\cup w^*$, and $M^q\geq
M^p+ |w^p|\cdot |w^*|>k$. Since $|w^q|>k$ we also have $n^q>k$ (remember
Observation \ref{obsbrick}(1)).
\end{proof}

\medskip

\begin{claim}
  \label{cl6}
 Assume that 
\begin{enumerate}
\item[(a)] $p\in\cP_\iota$ and $w\subseteq w^p$, $|w|\geq 3$, 
\item[(b)] $\pi_0:w\longrightarrow \omega$ is a quasi--embedding (into $S$)
and $\rng(\pi_0)\subseteq w^+\in [\omega]^{<\omega}$.   
\end{enumerate}
Then there are $q\in\cP_\iota$ and $\pi:w^+\longrightarrow w^q$ such that
\begin{itemize}
\item $p\curlyeqprec q$ and $\pi$ is a quasi--embedding, and 
\item $(\pi\circ\pi_0)(\gb(p)\rest w)\Subset \gb(q)\rest
  \big(\pi\big[\pi_0[w] \big]\big)$. 
\end{itemize}
\end{claim}

\begin{proof}[Proof of the Claim]
Since $S$ is cute we may find a quasi embedding $\pi:w^+\longrightarrow 
\omega$ such that $\rng(\pi)\cap w^p=\emptyset$. Apply Claim \ref{cl5} to
$w,w, \rng(\pi),\pi\circ \pi_0,\pi\circ \pi_0$ here standing for
$w,u,w^*,\pi_0,\pi_1$ there. (Note that the assumption \ref{cl5}(b) is
satisfied vacuously.)
\end{proof}  

\begin{claim}
  \label{cl9}
Assume that 
\begin{enumerate}
\item[(a)] $p\in\cP_\iota$, $u\subseteq w\subseteq w^p$, $3\leq |w|$,  
\item[(b)] $\bar{k}\big(v\cup\{d\}\big)\neq |d\cap 
  v|$ whenever $v\subseteq u$ and $d\in w\setminus u$ and 
  $\bar{r}\big(v\cup\{d\}\big)=0$, 
\item[(c)]   $\pi_0,\pi_1:w \longrightarrow \omega$ are quasi--embeddings 
  (into $S$) such that $\pi_0( a )=\pi_1( a )$  for $ a \in u$ and
  $\pi_0[w\setminus u]\cap \pi_1[w\setminus u]=\emptyset$. 
\end{enumerate}
Then there are $q\in \cP_\iota$ and a quasi--embedding
$\pi:\rng(\pi_0)\cup\rng(\pi_1) \longrightarrow w^q$ such that
$p\curlyeqprec q$ and
\[(\pi\circ\pi_0)(\gb(p)\rest w)\Subset \gb(q)\rest
  (\pi\circ\pi_0[w])\quad\mbox{ and }\quad
  (\pi\circ\pi_1)(\gb(p)\rest w)\Subset \gb(q)\rest (\pi\circ\pi_1[w]).\] 
\end{claim}

\begin{proof}[Proof of the Claim]
Using the cuteness of $S$ we first pick a quasi embedding $\pi^+:\rng(\pi_0)
\cup\rng(\pi_1)\longrightarrow \omega$ such that $\rng(\pi^+)\cap w^p
=\emptyset$. Then apply Claim \ref{cl5} to $u,w,\rng(\pi^+), \pi^+\circ
\pi_0, \pi^+\circ \pi_1$ here standing for $u,w,w^*,\pi_0,\pi_1$ there. 
\end{proof}
\medskip 

Using Claims \ref{cl14} (for (ii)), \ref{cl9} (for (iii)) and \ref{cl6} (for
(iv)), and employing a suitable bookkeeping device we may inductively choose
a sequence $\langle p_\ell:\ell<\omega\rangle \subseteq \cP_\iota$ such that 
\begin{enumerate}
\item[(i)] $p_\ell\curlyeqprec p_{\ell+1}$ for all $\ell<\omega$,
\item[(ii)] for every $k<\omega$ there is an $\ell<\omega$ such that $k\in
  w^{p_\ell}$, $n^{p_\ell}>k$, and $M^{p_\ell}>k$, 
\item[(iii)]  if (a)\quad $k<\omega$, $u\subseteq w\subseteq w^{p_k}$, $3\leq
  |w|$,   and 
\begin{enumerate}
\item[(b)] $\bar{k}\big(v\cup\{d\}\big)\neq |d\cap 
  v|$ whenever $v\subseteq u$ and $d\in w\setminus u$ and 
  $\bar{r}\big(v\cup\{d\}\big)=0$, 
\item[(c)]   $\pi_0,\pi_1:w \longrightarrow \omega$ are quasi--embeddings 
  (into $S$) such that $\pi_0( a )=\pi_1( a )$  for $ a \in u$ and
  $\pi_0[w\setminus u]\cap \pi_1[w\setminus u]=\emptyset$, 
\end{enumerate}
then there is an $\ell<\omega$ and a quasi--embedding
$\pi:\rng(\pi_0)\cup\rng(\pi_1) \longrightarrow w^{p_\ell}$ such
that 
\[(\pi\circ\pi_0)(\gb(p_k)\rest w)\Subset \gb(p_\ell)\rest
  (\pi\circ\pi_0[w])\quad\mbox{ and }\quad
  (\pi\circ\pi_1)(\gb(p_k)\rest w)\Subset \gb(p_\ell)\rest (\pi\circ\pi_1[w]),\] 
\item[(iv)] if (a)\quad $k<\omega$, $w\subseteq w^{p_k}$, $3\leq |w|$,  and 
\begin{enumerate}
\item[(b)] $\pi_0:w \longrightarrow \omega$ is a quasi--embedding 
and $\rng(\pi_0)\subseteq u\in [\omega]^{<\omega}$, 
\end{enumerate}
then there is an $\ell<\omega$ and a quasi--embedding
$\pi:u\longrightarrow w^{p_\ell}$ (into $S$) such that 
$(\pi\circ\pi_0)(\gb(p_k)\rest w)\Subset \gb(p_\ell)\rest
  (\pi\circ\pi_0[w])$.
\end{enumerate}
For $m<\omega$ let $T_m=\bigcup\{t^{p_\ell}_m:\ell<\omega\ \wedge\
m<M^{p_\ell}\}$ and note that each  $T_m$ is a subtree of ${}^{\omega>}2$
without terminal nodes. Let $\bar{T}= \langle T_m:m<\omega\rangle$ and let
$\bar{c}=\bigcup\limits_{\ell<\omega} \bar{d}^{p_\ell}$. One easily verifies that
$\bar{T}$ has $(\bar{c},S)$--controlled amalgamation property.
\bigskip

Let $\ndrk$ be the non-disjointness rank on $\Mtk$ (see Definitions
\ref{mtkDef}, \ref{ndrkdef}). Note that $\cM^{p_\ell}\subseteq \Mtk$ for
each $\ell<\omega$.   

\begin{claim}
\label{cl10}  
Assume that $N<\omega$, $w_0\subseteq w^{p_N}$, $3\leq |w_0|$, $\ell_0\leq
n^{p_N}$ and $\bn_0=\bmm^{\gb(p_N)}(\ell_0,w_0)\in \cM^{p_N}$. Then  
\begin{enumerate}
\item $\bar{r}(w_0)\leq \ndrk(\bn_0)$.  
\item If $|w_0|\geq 4$, then $\ndrk(\bn_0) \leq \omega\cdot
  \big(\bar{r}(w_0)+1\big)$  (ordinal product).   
\end{enumerate}
\end{claim}

\begin{proof}[Proof of the Claim]
(1)\quad By induction on $\alpha$ we show (for all $\bn_0,N,\ell_0,w_0$)
that $\alpha\leq \bar{r}(w_0)$ implies $\alpha\leq\ndrk(\bn_0)$. 

For the successor step, suppose that $\alpha+1\leq \bar{r}(w_0)$. Assume 
$\nu\in u_{\bn_0}$ and let $a\in w_0$ be such that
$\nu=\eta_a^{p_N}\rest\ell_0$. Let $L=|w_0|+1$, $\ell=|a\cap w_0|$ and let
$\varphi_0:w_0\longrightarrow L\setminus\{\ell+1\}$ and
$\varphi_1:w_0\longrightarrow L\setminus\{\ell\}$  be the increasing
bijections. Note  that
\begin{itemize}
\item $\varphi_0(x)=\varphi_1(x)$ whenever $x\in w_0\setminus\{a\}$, and
\item $\varphi_0^{-1}[u] =\varphi_1^{-1}[u]$ whenever $u\subseteq
  L\setminus\{ \ell,\ell+1\}$. 
\end{itemize}
We define a $\YZR(\vare)$--system $s=(X^s,\bar{r}^s,\bar{\jmath}^s,
\bar{k}^s)$ as follows. We set $X^s=L$ and for $u\subseteq X^s$ we put 
\begin{enumerate}
\item[$(\circledcirc)_1$] if $\ell\notin u$, then $\bar{r}^s(u)=\bar{r}
  \big(\varphi_1^{-1}[u]\big)$, $\bar{\jmath}^s(u)=\bar{\jmath}
  \big(\varphi_1^{-1}[u]\big)$ and $\bar{k}^s(u)=\bar{k}
  \big(\varphi_1^{-1}[u]\big)$, 
\item[$(\circledcirc)_2$] if $\ell+1\notin u$, then $\bar{r}^s(u)=\bar{r}
  \big(\varphi_0^{-1}[u]\big)$, $\bar{\jmath}^s(u)=\bar{\jmath}
  \big(\varphi_0^{-1}[u]\big)$ and $\bar{k}^s(u)=\bar{k}
  \big(\varphi_0^{-1}[u]\big)$, 
\item[$(\circledcirc)_3$]  if $\ell,\ell+1\in u$, then
  $\bar{r}^s(u)=\alpha$, $\bar{\jmath}^s(u)=\max\big(\bar{\jmath}(v):
  \emptyset\neq v\subseteq w_0\big)+1$, and $\bar{k}^s(u)=|\ell\cap u|$. 
\end{enumerate}
One easily verifies $(*)_1$, $(*)_2$ and $(*)_4$ of Definition
\ref{sysdef}. Concerning \ref{sysdef}$(*)_3$, we note that if $\emptyset
\neq u\subseteq w\subseteq X^s$ and $\{\ell,\ell+1\}\nsubseteq w$, then
$\bar{r}^s(u)\geq \bar{r}^s(w)$ by $(\circledcirc)_1+(\circledcirc)_2$ and
the properties of $\bar{r}$. If $\{\ell,\ell+1\}\subseteq w$, $\emptyset\neq
u\subseteq w$, then $\bar{r}^s(w)=\alpha$ and either $\bar{r}^s(u)=\alpha$
(when $\ell,\ell+1\in u$) or $\bar{r}^s(u)\geq\alpha+1$ (if $\{\ell,\ell+1\}
\nsubseteq u$). (Since $\bar{r}(w_0)\geq \alpha+1$, for every nonempty
$v\subseteq w_0$ we have $\bar{r}(v)\geq \alpha+1$.) Thus the only possibly
unclear demand is \ref{sysdef}$(*)_5$. So suppose that $\emptyset\neq
u\subseteq X^s$ and $u=\{a_0,\ldots,a_m\}$ is the increasing enumeration. We
want to argue that there is no $b\in X^s\setminus u$ such that
\begin{enumerate}
\item[$(\spadesuit)_b$] $|u\cap b|=\bar{k}^s(u)$, $\bar{\jmath}^s\big(
  (u\setminus \{a_{\bar{k}^s(u)}\})\cup\{b\}\big)=\bar{\jmath}^s(u)$ and
  $\bar{r}^s\big(u\cup\{b\}\big) = \bar{r}^s(u)$. 
\end{enumerate}
{\sc Case 1}:\quad $\ell\notin u$\\
By $(\circledcirc)_1$, there is no $b\in \big(X^s\setminus u\big) \setminus
\{\ell\}$ satisfying $(\spadesuit)_b$. If $\ell+1\in u$, then
$\bar{r}^s\big(u\cup\{\ell\} \big)=\alpha<\alpha+1\leq \bar{r}^s(u)$. Therefore,
$(\spadesuit)_\ell$ fails when $\ell+1\in u$. If $\ell+1\notin u$, then by
$(\circledcirc)_2$ the statement $(\spadesuit)_\ell$ must fail
too. Consequently, in the current case, there is no $b\in X^s\setminus u$
for which $(\spadesuit)_b$ holds true.\\
{\sc Case 2}:\quad $\ell+1\notin u$\\
Similar to Case 1, just interchanging $\ell$ and $\ell+1$.\\
{\sc Case 3}:\quad $\ell,\ell+1\in u$\\
Then, by $(\circledcirc)_3$, $a_{\bar{k}^s(u)}=\ell$ and for all $b\in
X^s\setminus u$ we get $\bar{\jmath}^s\big((u\setminus
\{\ell\})\cup\{b\}\big)< \bar{\jmath}^s(u)$. Thus $(\spadesuit)_b$ fails for
all $b\in X^s\setminus u$.
\bigskip

Since $S$ is cute, there is a quasi-embedding $\varphi:X^s\longrightarrow
\omega$ of $s$ into $S$ such that
$\max(w_0)<\min\big(\rng(\varphi)\big)$. Let $b=\varphi(\ell)$,
$b'=\varphi(\ell+1)$, $u=\rng(\varphi)\setminus \{b,b'\}$ and
$\pi_0=\varphi\circ \varphi_0$, $\pi_1=\varphi\circ \varphi_1$. Then 
\begin{itemize}
\item $u\cup\{b,b'\}\subseteq \omega \setminus w_0$, $b\neq b'$, and
   $\bar{r}\big(u\cup\{b,b'\}\big)=\alpha$, 
\item $\pi_0:w_0\longrightarrow u\cup\{b\}$ and  $\pi_1:w_0 \longrightarrow
  u\cup\{b'\}$ are quasi--embeddings and $\pi_0(a)=b$ and $\pi_1(a)=b'$, and
  $\pi_0\rest \big(w_0\setminus\{a\}\big)= \pi_1\rest
  \big(w_0\setminus\{a\}\big)$. 
\end{itemize}
Since $\bar{r}(w_0)\geq\alpha+1\geq 1$, we are sure that for every nonempty
$v\subseteq w_0$ we have $\bar{r}(v)\neq 0$. Therefore assumption (b) of
condition (iii) of the construction above is satisfied vacuously and we may
use that condition to claim that there are $K<\omega$ and a quasi--embedding 
$\pi: u\cup\{b,b'\}\longrightarrow w^{p_K}$ such that 
\[(\pi\circ\pi_0)(\gb(p_N)\rest w)\Subset \gb(p_K)\rest
  (\pi\circ\pi_0[w])\ \mbox{ and }\ 
  (\pi\circ\pi_1)(\gb(p_N)\rest w)\Subset \gb(p_K)\rest (\pi\circ\pi_1[w]).\] 
Note that $\bar{r}\big(\pi[u\cup\{b,b'\}]\big)=\alpha$. Let
$\bn_1=\bmm^{\gb(p_K)}(n^{p_K},\pi[u\cup\{b,b'\}])\in\cM^{p_K}\subseteq  
\Mtk$. Then $\bn_0\sqsubseteq \bn_1$, $2\leq |\{\eta\in u_{\bn_1}:
\nu\vtl\eta\}|$ and (by the inductive hypothesis)
$\alpha\leq\ndrk(\bn_1)$. 

Now we may conclude that $\ndrk(\bn_0)\geq\alpha+1$. The rest is clear. 
\bigskip

\noindent (2)\quad By induction on $\alpha$ we argue that
$\bar{r}(w_0)\leq \alpha$ implies $\ndrk(\bn_0)\leq \omega\cdot (\alpha+1)$
(for all  $\bn_0,N,\ell_0,w_0$).   

Assume first $\bar{r}(w_0)=0$ and let $a\in w_0$ be such that $|a \cap
w_0|=\bar{k}(w_0)$. Suppose that there is $\bmm\in\Mtk$ such
that  $\bmm\sqsupseteq \bn_0$ and  $2\leq |\{\nu\in u_\bmm:\eta^{p_N}_a\rest
\ell_0\vtl \nu\}|$. We may also demand that for some $K>N$ we have 
$\bmm\in {\mathbf   M}^{n^{p_K}}_{\bar{t}^{p_K},\iota}$ and
$\ell_\bmm=n^{p_K}$ and $|u_{\bmm}|=|u_{\bn_0}|+1\geq 5$.  Now use Claim
\ref{cl7} to find $\bn_1\in\cM^{p_K}$ and $\rho$ such that $\bn_1 
\doteqdot \bmm+\rho$. Let $\bn_1=\bmm^{\gb(p_K)}(n^{p_K},w_1)$. Let $b,b'\in   
w_1$ be such that   
\begin{enumerate}
\item[$(\clubsuit)$] $(\eta_a^{p_N}\rest \ell_0) +(\rho\rest \ell_0)=
  \eta_b^{p_k}\rest \ell_0=\eta_{b'}^{p_k}\rest \ell_0$, and $b\neq b'$.
\end{enumerate}
Then $\bmm_0\stackrel{\rm def}{=} \bmm^{\gb(p_K)}(\ell_0, w_1\setminus
\{b\}) \in \cM^{p_K}$ and $\bn_0+(\rho\rest \ell_0) \doteqdot \bmm_0$. 
By condition \ref{brick}$(\boxplus)_6$ for $p_K$ the order isomorphism
$\pi:w_0\longrightarrow w_1\setminus \{b\}$ is a quasi--embedding and
$(\eta_c^{p_K} \rest \ell_0)+(\rho\rest \ell_0)= \eta^{p_K}_{\pi(c)}\rest
\ell_0$. Therefore $\bar{r}(w_1\setminus \{b\})=0$ and $\bar{k}(w_1\setminus  
\{b\})=\bar{k}(w_0)=|a\cap w_0|=|b'\cap w_1|$. But then $\bmm_0\sqsubseteq
\bn$, $b,b'\in w_1$ and $(\clubsuit)$ contradict condition
\ref{brick}$(\boxplus)_7$. Consequently, $\ndrk(\bn_0)=0$.
\medskip

Assume now $0<\bar{r}(w_0)\leq \alpha$ (and the statement is true for ranks 
below $\alpha$). Let $a\in w_0$ be such that $|a \cap w_0|=\bar{k}(w_0)$ and
suppose that $\bn^*\in\Mtk$ satisfies
\[\bn_0\sqsubseteq \bn^*\ \mbox{ and }\ 
|u_{\bn^*}|=|u_{\bn_0}|+1\geq 5\ \mbox{ and }\ 2\leq |\{\nu\in
  u_{\bn^*}:\eta^{p_N}_a\rest \ell_0\vtl \nu\}|.\]  
We will argue that $\ndrk(\bn^*)<\omega\cdot (\alpha+1)$. So suppose this is not
the case and $\ndrk(\bn^*)\geq \omega\cdot \alpha+\omega$. Let $L$ be such
that $c_{h^{\bn^*}_i(\eta,\nu)}\leq L$ for all $(\eta,\nu)\in
\big(u_{\bn^*}\big)^{\langle 2\rangle}$ and $i<\iota$. Using Lemma
\ref{lemonrk}(8) and Observation \ref{stepup} we may find $\bn^+\in\Mtk$
such that
\[\bn^*\sqsubseteq \bn^+\ \mbox{ and }\ 
|u_{\bn^*}|=|u_{\bn^+}|\geq 5\ \mbox{ and }\ \ell_{\bn^+}>L\ \mbox{ and }\
\ndrk(\bn^+)\geq \omega\cdot \alpha+ 1170.\] 
Take $K>N+L$ such that $\bn^+\in {\mathbf
  M}^{n^{p_K}}_{\bar{t}^{p_K},\iota}$. Now we may find $\bn_1\in\cM^{p_K}$
which is essentially the same as a translation of $\bn^+$ (exists by Claim
\ref{cl7}), say $\bn_1= \bmm^{\gb(p_K)}(\ell_{\bn^+},w_1)$. Then for some
distinct $b,b'\in w_1$ we have   
\begin{itemize}
\item $\eta^{p_K}_b\rest \ell_0=\eta^{p_K}_{b'}\rest 
  \ell_0$, $\bmm^{\gb(p_K)}(\ell_0,w_1\setminus\{b\}),
  \bmm^{\gb(p_K)}(\ell_0,w_1\setminus\{b'\})\in \cM^{p_K}$   and  
\item $\bmm^{\gb(p_K)}(\ell_0,w_1\setminus\{b\}) \doteqdot
  \bmm^{\gb(p_K)}(\ell_0,w_1\setminus\{b'\})$, and  they are essentially the 
  same as a  translation of  $\bn_0$, and
\item the translation above maps $\eta^{p_K}_a\rest \ell_0\in
  u_{\bn_0}$ to $\eta^{p_K}_{b'}\rest \ell_0$ ($\eta^{p_K}_b\rest \ell_0$,
  respectively).   
\end{itemize}
By condition \ref{brick}$(\boxplus)_6$ for $p_K$ we know that the order
isomorphism from $w_0$ onto $w_1\setminus \{b\}$ ($w_1\setminus \{b'\}$,
respectively) is a quasi--embedding mapping $a$ onto $b'$ ($b$,
respectively). Therefore
\begin{itemize}
\item $\bar{r}(w_0)=\bar{r}(w_1\setminus\{b\})=\bar{r}(w_1\setminus
  \{b'\})$,  and  
\item $\bar{k}(w_0)=\bar{k}(w_1\setminus\{b\})=\bar{k}(w_1\setminus
  \{b'\})$, and 
\item $|w_1\cap b|=|w_1\cap b'|=|w_0\cap a|=\bar{k}(w_0)$. 
\end{itemize}
Since $\bar{r}(w_0)>0$, also $\bar{\jmath}(w_0)  =\bar{\jmath}(w_1
\setminus\{b\}) =\bar{\jmath}(w_1\setminus\{b'\})$. Therefore, by
\ref{sysdef}$(*)_5$, $\bar{r}(w_1) <\bar{r}(w_1\setminus\{b\})\leq \alpha$
and by the inductive hypothesis we get  
\[\ndrk(\bn^+)=\ndrk(\bn_1)\leq \omega\cdot (\bar{r}(w_1)+1)\leq \omega\cdot
  \alpha,\]
contradicting the choice of $\bn^+$. 
\medskip

Now we may conclude that  $\ndrk(\bn_0)\leq \omega\cdot (\alpha+1)$.   
\end{proof}

\begin{claim}
\label{cl11}  
$\vare\leq \ndrk(\bar{T})\leq\omega\cdot (\vare+2)+2$. 
\end{claim}

\begin{proof}[Proof of the Claim]
By the cuteness of $S$, there are $w\in[\omega]^{<\omega}$ with
$\bar{r}(w)=\vare$ and $|w|\geq 3$. Therefore Claim \ref{cl10}(1)
immediately implies the first inequality.

For the second inequality, suppose towards contradiction that $\bmm\in \Mtk$
is such that $\ndrk(\bmm)\geq\omega\cdot (\vare+2)+3$. Then we may pick
$\bn\in \Mtk$ such that  
\[\bmm\sqsubseteq \bn,\quad 5\leq |u_\bn|,\quad \mbox{ and }\quad 
  \ndrk(\bn)\geq\omega\cdot (\vare+2).\]
(Note that we coul dhave $|u_{\bmm}|=2$. To get $|u_\bn|\geq 5$ we may need
to drop the rank a little.) Let $L$ be such that $c_{h^{\bn}_i(\eta,\nu)}
\leq L$ for all $(\eta,\nu)\in\big(u_{\bn}\big)^{\langle 2\rangle}$ and
$i<\iota$. Like in the previous  Claim, use Observation \ref{stepup} to find
$\bn^+\in\Mtk$ such that  
\[\bn\sqsubseteq \bn^+\ \mbox{ and }\ 
|u_{\bn}|=|u_{\bn^+}|\geq 5\ \mbox{ and }\ \ell_{\bn^+}>L\ \mbox{ and }\ 
\ndrk(\bn^+)\geq \omega\cdot (\vare+1)+ 1170.\] 
Take an $N$ such that $\bn^+\in {\mathbf
  M}^{n^{p_N}}_{\bar{t}^{p_N},\iota}$. By Claim \ref{cl7} there are
$\rho\in {}^{\ell_{\bn^+}} 2$ and $\bn^*\in\cM^{p_N}$ such 
that  $(\bn^++\rho)\doteqdot \bn^*$.  But now by Claim \ref{cl10}(2) and
Lemma \ref{lemonrk}(4) we have $\ndrk(\bn^+)=\ndrk(\bn^*)\leq\omega\cdot 
(\vare+1)$, a contradiction. 
\end{proof}
\end{proof}

\section{Conclusions and Questions}

For a countable ordinal $\vare>0$ and $2\leq \iota<\omega$ let
$\bar{T}^{\vare,\iota}=\langle T_m^{\vare,\iota}:m<\omega\rangle$ be the
sequence of trees given by Theorem \ref{rankset} (for some $S$ and $\bar{c}$
as there).   Let $B_{\vare,\iota}=\bigcup\limits_{m<\omega}
\lim(T_m^{\vare,\iota})$. 

\begin{corollary}
If  $\lambda$ is a cardinal such that ${\rm NPr}^\vare(\lambda)$ holds true, 
then there exist a ccc forcing notion $\bbP$ such that    
\[\begin{array}{l}
\forces_{\bbP}\mbox{`` there is a sequence
    }\langle\rho_\alpha:\alpha<\lambda\rangle
\mbox{ of distinct elements of $\can$ such that}\\
\qquad \big|(\rho_\alpha+B_{\vare,\iota})\cap
    (\rho_\beta+B_{\vare,\iota})\big|\geq 2\iota\mbox{ for all
    }\alpha,\beta<\lambda\\
    \qquad \mbox{but there is no perfect set of such $\rho$'s. ''} 
\end{array}\]
\end{corollary}

\begin{proof}
  Since $\bar{T}^{\vare,\iota}$ has the $(\bar{c},S)$-controlled
  amalgamation property and ${\rm NPr}^\vare(\lambda)$ holds true, we may
  use Theorem \ref{force} to find a ccc forcing notion $\bbP$ which adds
  $\lambda$ many translations of $B_{\vare,\iota}$ with intersections of
  size at least $2\iota$. On the other hand, since
  $\ndrk(\bar{T}^{\vare,\iota}) <\omega_1$ and the rank $\ndrk$ is absolute,
  it follows from Proposition \ref{eqnd} that in no extension there is
  perfect set of such translations.
\end{proof}

\begin{corollary}
  Assume MA and $\aleph_\vare<\con$.  Then
  \begin{itemize}
  \item   there is a sequence $\langle\rho_\alpha:\alpha<\aleph_\vare
    \rangle$ of distinct elements of $\can$ such that for
    $\alpha,\beta<\aleph_\vare$ 
\[\big|(\rho_\alpha+B_{\vare,\iota})\cap 
  (\rho_\beta+B_{\vare,\iota})\big|\geq 2\iota,\]
\item for every perfect set $P\subseteq \can$ there are $\eta,\nu\in P$ such
  that
\[\big|(\eta+B_{\vare,\iota})\cap  (\nu+B_{\vare,\iota})\big|< 2\iota.\]
  \end{itemize}
\end{corollary}

\begin{proof}
It follows from Proposition \ref{cl1.7-522} that ${\rm
  NPr}^\vare(\aleph_\vare)$ holds true. Hence Theorem \ref{force} gives us a
ccc  forcing notion $\bbP$ for $\aleph_\vare$. Applying MA for $\bbP$ and
its dense subsets $D_\alpha,D^n$ (for $\alpha<\aleph_\vare$ and $n<\omega$;
see Claim \ref{cl2}) we get the required sequence $\langle
\rho_\alpha:\alpha<\aleph_\vare\rangle$. Nonexistence of a perfect set of
$2\iota$--non-disjoint translations is the consequence of
$\ndrk(\bar{T}^{\vare,\iota})<\omega_1$. 
\end{proof}

\begin{corollary}
  \label{answ}
  There exists a sequence
  $\langle\eta_\alpha:\alpha<\omega_1\rangle$ of distinct elements of $\can$
  such that $\big|(\rho_\alpha+B_{\vare,\iota})\cap
  (\rho_\beta+B_{\vare,\iota})\big|\geq 2\iota$ for   all
  $\alpha,\beta<\omega_1$, but there is no perfect set of such $\eta$'s. 
\end{corollary}

\begin{proof}
Since $\ndrk\big(\bar{T}^{\vare,\iota}\big)<\omega_1$ we know that there is
no perfect set $P\subseteq \can$ with the property that
$\big|(\rho_0+B_{\vare,\iota})\cap (\rho_1+B_{\vare,\iota})\big|\geq 2\iota$
for all $\rho_0,\rho_1\in P$. On the other hand, by Theorem \ref{force},
there is a ccc forcing notion forcing that
\begin{enumerate}
\item[$(\divideontimes)$]  ``there is a sequence
  $\langle\eta_\alpha:\alpha<\omega_1\rangle$ of distinct 
elements of $\can$ such that $\big|(\rho_\alpha+B_{\vare,\iota})\cap
(\rho_\beta+B_{\vare,\iota})\big|\geq 2\iota$ for all
$\alpha,\beta<\omega_1$''. 
\end{enumerate}
By Keisler's Completeness Theorem for logic with the quantifier ``there
exists uncountably many'' \cite[Theorem 4.10]{Ke70}, the assertion in
$(\divideontimes)$ is absolute between $\bV$ and it ccc forcing
extensions. For the sake of completeness, let us present this argument in
some detail. 

The spectrum of translation $2\iota$--non-disjointness
\[\stnd_{2\iota}(B_{\vare,\iota})\stackrel{\rm
    def}{=}\{(x,y)\in\can\times\can: |(B_{\vare,\iota}+x)\cap
  (B_{\vare,\iota}+y)|\geq 2\iota\}\]
of the $\Sigma^0_2$ set $B_{\vare,\iota}$ is a $\Sigma^0_2$ subset of the
product $\can\times\can$. Let $F^*_m\subseteq \can\times\can$ be closed sets
such that $\stnd_{2\iota}(B_{\vare,\iota})=\bigcup\limits_{m<\omega} F^*_m$
and let
\[S^*_m=\big\{(x\rest n,y\rest n): (x,y)\in F^*_m\ \&\
  n<\omega\big\}.\]
The assertion $(\divideontimes)$ means that there is an uncountable set
$D\subseteq \can$ such that $D\times D\subseteq
\stnd_{2\iota}(B_{\vare,\iota})$. To show the downward absoluteness of
$(\divideontimes)$, we will essentially repeat the arguments in the proof
of Kubi\'s and Shelah \cite[Proposition 3.2]{KbSh:802} (following also the
lines presented in Farah, Hru\v{s}\'{a}k and Mart\'{i}nez \cite{FHM05}).   

The language $\cL_{\omega_1\omega}$  (see Keisler \cite{Ke71}) is formed
from the first order $\cL$ by allowing countable infinite conjunctions, but
only finite quantifiers. Then the language $\cL_{\omega_1\omega}(Q)$  is
formed by adding to $\cL_{\omega_1\omega}$  the additional quantifier $(Qx)$
with the meaning ``there are uncountably many $x\ldots$''. The axioms for
$\cL_{\omega_1\omega}(Q)$ include:
\begin{itemize}
\item all the axiom schemes for $\cL$,
\item $\neg(Qx)(x= y\vee x= z)$,
\item $(\forall x)(\varphi\Rightarrow\psi)\Rightarrow \big((Qx)\varphi
  \Rightarrow (Qx)\psi\big)$\qquad
  where $\varphi,\psi$ are formulas,
\item $(Qx)\varphi(x\ldots)\Leftrightarrow (Qy) \varphi(y\ldots)$\qquad
where $\varphi(x\ldots)$ is a formula in which $y$ does not occur and
$\varphi(y\ldots)$ is obtained by repalcing each free occurence of $x$ by
$y$,
\item $(Qy)(\exists x)\varphi\Rightarrow \Big((\exists x)(Qy)\varphi\vee
  (Qx)(\exists y)\varphi\Big)$\qquad
where $\varphi$ is a formula, 
\item $\bigwedge\limits_{n<\omega} (\varphi\Rightarrow \psi_n) \Rightarrow
  (\varphi \Rightarrow \bigwedge\limits_{n<\omega} \psi_n)$ 
\item $\bigwedge\limits_{n<\omega} \psi_n\Rightarrow \psi_m$ \quad (for
  $m<\omega$),
\item $(Qx)\bigvee\limits_{n<\omega} \psi_n\Rightarrow
  \bigvee\limits_{n<\omega} (Qx)\psi_n$\quad where
  $\bigvee\limits_{n<\omega} \chi_n$ is an abbreviation for
  $\neg\bigwedge\limits_{n<\omega} \neg\chi_n$.  
\end{itemize}
The rules of inference are {\em modus ponens}, {\em generalization} and the
{\em rule of conjunction}:

\quad from $\psi_0,\psi_1,\psi_2,\ldots$ infer $\bigwedge\limits_{n<\omega}
\psi_n$. 

\noindent A standard model in $\cL_{\omega_1\omega}(Q)$ is a model $\bbM$ of
the first order language of $\cL$, where the notion of satisfaction is
defined in natural way with the $(Qx)$ clause in the definition being 

\quad $\bbM\models (Qv_m)\varphi[a_1,\ldots,a_n]$\quad if and only if 

\quad the set $\{b\in M:\bbM\models \varphi[a_1,\ldots,
a_{m-1},b,a_{m+1},\ldots, a_n]\}$ is uncountable.

\noindent The Completeness Theorem for  $\cL_{\omega_1\omega}(Q)$
\cite[Theorem 4.10]{Ke70} asserts that for a language
$\cL_{\omega_1\omega}(Q)$ with a countable vocabulary, {\em a sentence
  $\varphi$ of $\cL_{\omega_1\omega}(Q)$ is consistent if and only if
  $\varphi$ has a standard model.}  

We shall use countably many predicates and function symbols and the
quantifier $Q$ to describe an uncountable set $D$ of pairwise
$2\iota$-non-disjoint translations of
$B_{\vare,\iota}=\bigcup\limits_{m<\omega} \lim(T^{\vare,\iota}_m)$, i.e.,
$D\subseteq \stnd_{2\iota}(B_{\vare,\iota})$.

Let $C$ and $D$ be unary predicates intentionally describing the Cantor
space and its uncountable subset $D$. Constants 0,1 will be used in the
usual way and unary function symbols $P_k(x)$ will mean ``the $k$-th
coordinate of $x$'' (for $k<\omega$). Now we define a sentence  $\theta$
in the language $\cL_{\omega_1\omega}(Q)$ for the vocabulary
$\{C,D,0,1,P_k\}_{k<\omega}$.

First, $\theta^*$ is the conjunction of the following two sentences:
\begin{itemize}
\item $\big(\forall x\big)\big(C(x)\Rightarrow \bigwedge\limits_{k<\omega}(
  P_k(x)=0\vee P_k(x)=1)\big)$,
\item $\big(\forall x,y\big)\big((C(x)\wedge C(y)\wedge x\neq y)\Rightarrow
  \bigvee\limits_{k<\omega} P_k(x)\neq P_k(y) \big)$. 
\end{itemize}
Now, for finite sequences $\rho,\sigma\in {}^n 2$, a formula
$\phi^{\sigma,\rho}(x,y)$ is
\[C(x)\wedge C(y)\wedge \bigwedge\limits_{k<n}\big( P_k(x)=\sigma(k)\wedge
  P_k(y)=\rho(k)\big)\]
and then for $n,m<\omega$, a formula $\psi^n_m(x,y)$ is
\[C(x)\wedge C(y)\wedge \bigvee\limits_{(\sigma,\rho)\in S^*_m\cap
    ({}^n2\times {}^n 2)} \varphi^{\sigma,\rho}(x,y).\]
Finally, the formula $\psi_m(x,y)$ is $\bigwedge\limits_{n<\omega}
\psi^n_m(x,y)$. Now let $\theta$ be the sentence
\[\theta^*\wedge (Qx)D(x)\wedge (\forall x)(D(x)\Rightarrow C(x))\wedge
    \big(\forall x,y\big)\big((D(x)\wedge D(y))\Rightarrow
    \bigvee\limits_{m<\omega} \psi_m(x,y)\big).\]
It should be clear that every standard model of $\theta$ determines an
uncountable square included in $\stnd_{2\iota}(B_{\vare,\iota})$ (and vice
versa). So we may argue for the absoluteness of $(\divideontimes)$ as
follows.

Suppose $(\divideontimes)$ holds in $\bV[G]$. Then, in $\bV[G]$, the
sentence $\theta$ has a standard model, so it is consistent. But this means
that $\theta$ must be consistent in $\bV$, as otherwise we would have a
proof of $\neg\theta$ in $\bV$. Such a proof, while possibly infinite, can
be thought of as a labeled well founded tree. Therefore the same proof would
work in $\bV[G]$ as well, giving a contradiction. Consequently, the sentence
$\theta$ is consistent in $\bV$ and it has a standard model in $\bV$. Hence
$(\divideontimes)$ holds in $\bV$.
\end{proof}
\bigskip

The results presented in this paper leave several natural questions
open. First of all,

\begin{problem}
  What is the value of $\ndrk\big(\bar{T}^{\vare,\iota}\big)$ ?
\end{problem}

A natural question is if we can replace the amalgamation property in Theorem
\ref{force} with a requirement on the rank $\ndrk(\bar{T})$. In the
strongest form this would be the following question.

\begin{problem}
  Suppose that
  \begin{enumerate}
\item[(a)] $T_m\subseteq {}^{\omega>} 2$ (for $m<\omega$) are trees with no  
  maximal nodes, $\bar{T}=\langle T_m:m<\omega\rangle$, and 
\item[(b)]   $\ndrk$ is the non-disjointness rank on $\Mtk$, $2\leq
  \iota<\omega$,
\item[(c)]  $\vare\leq \ndrk(\bar{T})$, and $\lambda$ is a cardinal such
  that ${\rm NPr}^\vare(\lambda)$ holds true, 
\item[(d)] $B=\bigcup\limits_{m<\omega} \lim(T_m)$.
\end{enumerate}
Does there exist a ccc forcing notion $\bbP$ of size $\lambda$ such that    
\[\begin{array}{l}
\forces_{\bbP}\mbox{`` there is a sequence
    }\langle\eta_\alpha:\alpha<\lambda\rangle
\mbox{ of distinct elements of $\can$ such that}\\
\qquad \big|(\rho_\alpha+B)\cap
(\rho_\beta+B)\big|\geq 2\iota\mbox{ for all }\alpha,\beta<\lambda\mbox{
    ''\ \ ?} 
\end{array}\]
\end{problem}

The relevance of $\iota$ is yet to be discovered:

\begin{problem}
Does there exist a sequence $\bar{T}=\langle T_m:m<\omega\rangle$ of trees
$T_m\subseteq {}^{\omega>} 2$ (for $m<\omega$) such that for some
$2\leq \iota<\iota'<\omega$ we have $\ndrk(\bar{T})\neq {\rm
  ndrk}_{\iota'}(\bar{T})$ ? 
\end{problem}

Of course, the next steps could be to investigate $\stnd_\omega$ and
$\stnd_{\omega_1}$:

\begin{problem}
Is is consistent to have a Borel set $B\subseteq \can$ such that 
\begin{itemize}
\item for some uncountable set $H$, $(B+x)\cap (B+y)$ is uncountable for
  every $x,y\in H$, but 
\item for every perfect set $P$ there are $x,y\in P$ with $(B+x)\cap (B+y)$
  countable?
\end{itemize}
Similarly if ``uncountable / countable'' are replaced with ``infinite /
finite'', respectively. 
\end{problem}

As mentioned before, our arguments relay on the algebraic properties of 
$\can$. So, one should ask for the following. (The constructions in
\cite{RoSh:1187} might be relevant here.) 

\begin{problem}
Generalize the results of this paper (Theorems \ref{force} and 
\ref{rankset}) to the case of Polish groups (not just $\can$).   
\end{problem}

Hopefully, the investigations of $\stnd$ will shed some light on the dual
case of $\std_\kappa$. In particular:

\begin{problem}
Is it consistent to have a Borel set $B\subseteq \can$ such that 
\begin{itemize}
\item $B$ has uncountably many pairwise disjoint translations, but
\item there is no perfect set of pairwise disjoint translations of $B$ ? 
\end{itemize}
\end{problem}

Finally, let us recall the big question concerning the ``cutting point'' in
this considerations. let  $\lambda_{\omega_1}$ be the smallest $\lambda$
such that ${\rm Pr}^\vare(\lambda)$ holds true for all $\vare<\omega_1$ (see
Definition \ref{PRdef}).

\begin{problem}
Is $\lambda_{\omega_1}=\aleph_{\omega_1}$ ?  
\end{problem}

\end{document}